\documentclass[a4paper,12pt,reqno]{amsart}
\usepackage[utf8]{inputenc}
\usepackage[english]{babel}
\usepackage[a4paper,twoside,top=1in, bottom=0.9in, left=0.7in, right=0.7in]{geometry} 
\usepackage{graphicx}
\usepackage{hyperref}
\usepackage{amsmath}
\usepackage{amsthm}
\usepackage{amssymb}
\usepackage{amsfonts}
\usepackage{esint} 
\usepackage{mathrsfs} 
\usepackage{xcolor}
\usepackage{array}
\usepackage{hhline}
\usepackage{enumerate} 
\usepackage{enumitem} 
\usepackage{xparse} 
\usepackage{mathtools}
\usepackage{float}
\usepackage{comment}
\usepackage{nicefrac}
\usepackage{graphicx}

\definecolor{newblue}{rgb}{0.2, 0.3, 0.85}
\hypersetup{colorlinks=true, linkcolor=newblue, citecolor=newblue, urlcolor  = newblue} 

\newcommand{\mm}{\mathfrak m}
\newcommand{\sfd}{{\sf d}}

\newcommand{\nchi}{{\raise.3ex\hbox{\(\chi\)}}}

\usepackage{fancyhdr} 
\usepackage{ifthen} 
\usepackage{forloop} 
\usepackage{xstring} 
\usepackage{blindtext} 
\usepackage{nomencl}
\usepackage{dsfont} 
\usepackage{faktor} 
\usepackage{tikz}
\usetikzlibrary{arrows}

\numberwithin{equation}{section}

\usepackage[english,capitalize]{cleveref}

\usepackage[maxbibnames=99,backend=biber, doi=false, url=false, backref=true]{biblatex}
\usepackage[colorinlistoftodos,textwidth=2.3cm]{todonotes}

\definecolor{dgreen}{rgb}{0.0, 0.56, 0.0}

\newcommand{\N}{\ensuremath{\mathbb N}}
\newcommand{\R}{\ensuremath{\mathbb R}}
\newcommand{\meas}{\mathfrak{m}}

\newcommand{\lip}{{\rm lip \,}}

\newcommand{\st}{\ensuremath{\ :\ }} 
\newcommand{\eqdef}{\ensuremath{\vcentcolon=}}
\newcommand \eps{\ensuremath{\varepsilon}} 
\renewcommand{\epsilon}{\varepsilon}

\newcommand{\de}{\ensuremath{\,\mathrm d}} 
\DeclareMathOperator{\supp}{spt} 

\newcommand{\CD}{\mathsf{CD}}
\newcommand{\RCD}{\mathsf{RCD}}
\newcommand{\dist}{\mathsf{d}} 
\newcommand{\haus}{\mathcal{H}} 

\DeclareMathOperator{\vol}{vol}

\theoremstyle{plain}
\newtheorem{theorem}{Theorem}[section] 
\theoremstyle{plain}

\theoremstyle{plain}
\newtheorem{proposition}[theorem]{Proposition}
\theoremstyle{plain}
\newtheorem{lemma}[theorem]{Lemma}
\theoremstyle{plain}
\newtheorem{corollary}[theorem]{Corollary}
\theoremstyle{definition}
\newtheorem{definition}[theorem]{Definition} 
\theoremstyle{definition}
\newtheorem{remark}[theorem]{Remark}
\theoremstyle{definition}

\theoremstyle{definition}

\title[The isoperimetric problem via direct method]{The isoperimetric problem via direct method in noncompact metric measure spaces with lower Ricci bounds}

\author[G. Antonelli]{Gioacchino Antonelli}\address{Scuola Normale Superiore, Piazza dei Cavalieri, 7, 56126 Pisa, Italy.}\email{gioacchino.antonelli@sns.it}

\author[S. Nardulli]{Stefano Nardulli}\address{Centro de Matemática Cognição e Computação, Universidade Federal do ABC, Av. Dos Estados 5001, Santo André, SP, Brazil}\email{stefano.nardulli@ufabc.edu.br}

\author[M. Pozzetta]{Marco Pozzetta}\address{Dipartimento di Matematica e Applicazioni, Universit\`a di Napoli Federico II, Via Cintia, Monte S. Angelo, 80126 Napoli, Italy.}\email{marco.pozzetta@unina.it}

\date{\today}

\subjclass{Primary: 49Q20, 49J45, 53A35. Secondary: 53C23.}
\keywords{Isoperimetric problem, existence, direct method, Ricci lower bounds, $\RCD$ spaces.}

\begin{document}

\maketitle

\begin{abstract}
We establish a structure theorem for minimizing sequences for the isoperimetric problem on noncompact $\RCD(K,N)$ spaces $(X,\dist,\haus^N)$. Under the sole (necessary) assumption that the measure of unit balls is uniformly bounded away from zero, we prove that the limit of such a sequence is identified by a finite collection of isoperimetric regions possibly contained in pointed Gromov--Hausdorff limits of the ambient space $X$ along diverging sequences of points. The number of such regions is bounded linearly in terms of the measure of the minimizing sequence.

The result follows from a new generalized compactness theorem, which identifies the limit of a sequence of sets $E_i\subset X_i$ with uniformly bounded measure and perimeter, where $(X_i,\dist_i,\haus^N)$ is an arbitrary sequence of $\RCD(K,N)$ spaces.

An abstract criterion for a minimizing sequence to converge without losing mass at infinity to an isoperimetric set is also discussed. The latter criterion is new also for smooth Riemannian spaces.
\end{abstract}

\tableofcontents

\vspace{1cm}

\section{Introduction}

The \emph{isoperimetric problem} can be formulated on every ambient space possessing notions of \emph{volume measure} $\meas$ and \emph{perimeter} $P$ on (some subclass of) its subsets. Among sets having assigned positive volume, the problem deals with finding those having least perimeter. Among the most basic questions in the context of the isoperimetric problem, one would naturally ask whether there exist minimizers, called \emph{isoperimetric regions} (or \emph{isoperimetric sets}), but also what goes wrong in the minimization process in case such minimizers do not exist. The value of the infimum of the perimeter among sets of a given volume $V$ is called \emph{isoperimetric profile} at $V$.

A natural class of spaces where to set the isoperimetric problem is given by metric measure spaces $(X,\dist,\meas)$ (see \cref{sec:Preliminaries}). Indeed, the nonnegative Borel measure $\meas$ plays the role of a volume functional, and, together with a distance $\dist$, it is possible to give a definition of perimeter $P$ (see \cref{def:BVperimetro}). The smooth and more classical counterpart of these spaces is given by Riemannian manifolds. If $(M,g)$ is a Riemannian manifold of dimension $N$, the natural Riemannian distance and the $N$-dimensional Hausdorff measure $\haus^N$ yield the structure of metric measure space, and the corresponding definition of perimeter recovers the classical well-known notion à la Caccioppoli--De Giorgi. In fact, the theory of $BV$ functions and of the perimeter functional on metric measure spaces has been blossoming in the last decades \cite{Miranda03, Ambrosio02, AmbrosioAhlfors, MirandaPallaraParonettoPreunkert07, AmbrosioDiMarino14}.

We are here interested in the problem of the existence of isoperimetric regions, assuming least possible hypotheses on the ambient space. Hence the most natural way to approach the problem is to argue \emph{by direct method}, that is, by studying the behavior of a minimizing sequence of sets $E_i$ of fixed volume $V$ whose perimeter is converging to the isoperimetric profile at volume $V$. It is therefore understood that, by usual precompactness and lower semicontinuity, the problem of existence is nontrivial only in case the ambient is \emph{noncompact} (actually, with infinite measure), which we will always assume.

Already in the smooth ambient, the development of an effective theory of a direct method for the isoperimetric problem is a difficult task. Studying the problem in Euclidean solid cones, in \cite{RitRosales04} the authors identified a general mass splitting phenomenon of a minimizing sequence for the problem, where the sequence decomposes into two components, one converging in the space and the other diverging at infinity. Such result is here generalized in \cref{thm:RitoreRosalesNonSmooth}. Combining this approach with a concentration-compactness argument, in \cite{Nar14} the author performed a better description of the possible mass lost at infinity for the problem on Riemannian manifolds satisfying some mild asymptotic hypotheses on their ends. The theory has then been successfully applied to get existence theorems in \cite{MondinoNardulli16} and further generalized in \cite{FloresNardulliCompactness}. Finally, in \cite{AFP21} the above mentioned hypotheses on the ends of the manifold have been removed, further generalizing the method on Riemannian manifolds.

After the results in \cite{AFP21}, it is today understood that, already in smooth Riemannian manifolds, the isoperimetric problem becomes trivial, i.e., the isoperimetric profile vanishes, unless it is assumed a lower bound on the Ricci curvature and a positive lower bound on the volume of unit balls. In fact, as one of the two hypotheses is not satisfied, one can find examples where a description of the behavior of minimizing sequences is actually compromised, see \cite{AFP21}. Therefore, it becomes natural to consider the isoperimetric problem on $\RCD(K,N)$ metric measure spaces, which are spaces encoding synthetic notions of Ricci curvature bounded below by $K\in \R$ and dimension bounded above by $N\in(0,+\infty]$, see \cref{sec:Preliminaries}. Moreover, we shall address in this work the case of $\RCD(K,N)$ spaces of the form $(X,\dist,\haus^N)$, i.e., endowed with the $N$-dimensional Hausdorff measure. We will call such spaces $N$-dimensional $\RCD(K,N)$ spaces. The case of arbitrary volume measures $\meas$ appears to be more involved and related to a better understanding of the properties of the density of $\meas$ with respect to the Hausdorff measure of the essential dimension of the space. We stress that the class of $N$-dimensional $\RCD(K,N)$ spaces, that has been recently introduced and studied in the works \cite{Kitabeppu17, DePhilippisGigli18, AntBruSem19, BrueNaberSemola20}, is the non-smooth generalization of the class of non collapsed Ricci limit spaces \cite{ChCo1}, in which a volume convergence theorem holds. The Riemannian assumption is necessary here to exploit the convergence and stability results of \cite{AmborsioBrueSemola19}.

It is remarkable to notice that the development of a theory on such nonsmooth spaces already is a \emph{necessary consequence} also of the approach by direct method of the problem on perfectly smooth Riemannian manifolds \cite{AFP21}. Indeed, nonsmooth $\RCD(K,N)$ spaces $(X,\dist,\haus^N)$ arise as limits in pointed Gromov--Hausdorff sense (\cref{def:GHconvergence}) of smooth manifolds $M$ with Ricci and volume of unit balls  bounded below along sequences of points diverging on $M$.

Capitalizing on the methods developed in \cite{Nar14, AFP21, AntonelliPasqualettoPozzetta}, we are able to give a description of the behavior of perimeter minimizing sequences for the isoperimetric problem on $\RCD$ spaces as follows, generalizing and improving our previous results. For the notation and the notions of convergence appearing in \cref{thm:MassDecompositionINTRO}, see \cref{sec:Preliminaries}, \cref{def:GHconvergence}, and \cref{def:L1strong} below.

\begin{theorem}[Asymptotic mass decomposition]\label{thm:MassDecompositionINTRO}
Let $K\leq 0$ and $N\geq 2$.
Let $(X,\dist,\mathcal{H}^N)$ be a noncompact $\RCD(K,N)$ space. Assume there exists $v_0>0$ such that $\mathcal{H}^N(B_1(x))\geq v_0$ for every $x\in X$. Let $V>0$. For every minimizing $($for the perimeter$)$ sequence of bounded sets $\Omega_i\subset X$ of volume $V$, up to passing to a subsequence, there exist a nondecreasing bounded sequence $\{N_i\}_{i\in\mathbb N}\subseteq \mathbb N$, disjoint finite perimeter sets $\Omega_i^c, \Omega_{i,j}^d \subset \Omega_i$, and points $p_{i,j}$, with $1\leq j\leq N_i$ for any $i$, such that the following claims hold
\begin{itemize}
    \item $\lim_{i} \dist(p_{i,j},p_{i,\ell}) = \lim_{i} \dist(p_{i,j},o)=+\infty$, for any $j\neq \ell \le\overline N$ and any $o\in X$, where $\overline N:=\lim_i N_i <+\infty$;
    \item $\Omega_i^c$ converges to $\Omega\subset X$ in the sense of finite perimeter sets, $\mathcal{H}^N(\Omega_i^c)\to_i \mathcal{H}^N(\Omega)$, and $ P( \Omega_i^c) \to_i P(\Omega)$. Moreover $\Omega$ is an isoperimetric region in $X$;
    \item for every $0<j\le\overline N$, $(X,\dist,\mathcal{H}^N,p_{i,j})$ converges in the pmGH sense  to a pointed $\RCD(K,N)$ space $(X_j,\dist_j,\mathcal{H}^N,p_j)$. Moreover there are isoperimetric regions $Z_j \subset X_j$ such that $\Omega^d_{i,j}\to_i Z_j$ in $L^1$-strong and $P(\Omega^d_{i,j}) \to_i P(Z_j)$;
    \item it holds that
    \begin{equation}\label{eq:UguaglianzeIntro}
    I_X(V) = P(\Omega) + \sum_{j=1}^{\overline{N}} P (Z_j),
    \qquad\qquad
    V=\mathcal{H}^N(\Omega) +  \sum_{j=1}^{\overline{N}} \mathcal{H}^N(Z_j).
    \end{equation}
\end{itemize}
\end{theorem}

\cref{thm:MassDecompositionINTRO} states a general behavior for minimizing sequences of the isoperimetric problem. Roughly speaking, the mass of a sequence splits into at most finitely many pieces and it is totally recovered by finitely many isoperimetric regions sitting in spaces ``located at infinity'' with respect the original ambient space. Notice that \cref{thm:MassDecompositionINTRO} is not an existence theorem, nor it is a nonexistence result, instead it is a general tool for treating the problem by direct method. With such theorem it is then possible to recover the main existence and nonexistence results previously proved in \cite{MondinoNardulli16, AFP21, AntonelliBrueFogagnoloPozzetta2021}, and, actually, to suitably extend those to the nonsmooth $\RCD$ setting. We also mention that, taking into account \cref{prop:ProfileDecomposition} below, \eqref{eq:UguaglianzeIntro} can be thought as a generalized existence result of isoperimetric regions.

\medskip

The asymptotic mass decomposition result above actually follows from the following new result of generalized compactness of sequences of sets with uniformly bounded volume and perimeter. Notice that a generalized compactness result like the following one has been proved in \cite{FloresNardulliCompactness} in the Riemannian setting.

\begin{theorem}[Generalized compactness]\label{thm:GeneralizedCompactness}
Let $K\in \R$ and $N\geq 2$.
Let $(X_i,\dist_i,\mathcal{H}^N)$ be a sequence of $\RCD(K,N)$ spaces, and let $E_i \subset X_i$ be bounded sets of finite perimeter such that $\sup_i P(E_i)+ \haus^N(E_i)<+\infty$. Assume there exists $v_0>0$ such that $\mathcal{H}^N(B_1(x))\geq v_0$ for every $x\in X_i$, and for every $i$.

Then, up to subsequence, there exist a nondecreasing, possibly unboundend, sequence $\{N_i\}_{i\in\mathbb N}\subseteq \mathbb N_{\ge 1}$, points $p_{i,j} \in X_i$, with $1\le j\le N_i$ for any $i$, and pairwise disjoint subsets $E_{i,j}\subset E_i$ such that
\begin{itemize}
    \item $\lim_{i} \dist_i(p_{i,j},p_{i,\ell}) =+\infty$, for any $j\neq \ell<\overline N+1$, where $\overline N:=\lim_i N_i \in \N\cup\{+\infty\}$;
    
    \item For every $1\le j< \overline N+1$, the sequence $(X_i,\dist_i,\mathcal{H}^N,p_{i,j})$ converges in the pmGH sense to a pointed $\RCD(K,N)$ space $(Y_j,\dist_{Y_j},\mathcal{H}^N,p_j)$ as $i\to+\infty$;
    
    \item there exist sets $F_j\subset Y_j$ such that $E_{i,j}\to_i F_j$ in $L^1$-strong and there holds
    \begin{equation}\label{eq:GeneralizedCMPcontM}
        \lim_i \haus^N(E_i) = \sum_{j=1}^{\overline{N}} \haus^N(F_j),
    \end{equation}

    \begin{equation}\label{eq:GeneralizedCMPsemicontP}
        \sum_{j=1}^{\overline{N}} P(F_j) \leq \liminf_i P(E_i).
    \end{equation}
\end{itemize}
Moreover, if $E_i$ is an isoperimetric set in $X_i$ for any $i$, then $F_j$ is an isoperimetric set in $Y_j$ for any $j<\overline{N}+1$ and
\begin{equation}\label{eq:GeneralizedCMPcontP}
    P(F_j) = \lim_i P(E_{i,j}),
\end{equation}
for any $j<\overline{N}+1$.
\end{theorem}

We notice that, in the first part of \cref{thm:GeneralizedCompactness}, no a priori minimality property is required on the sequence of sets $E_i$. \cref{thm:GeneralizedCompactness} is clearly mostly useful in case the spaces $X_i$ in the statement are noncompact. Nevertheless, the statement applies also in case the diameter of the $X_i$'s is uniformly bounded, in which case one trivially finds $N_i=\overline{N}=1$ for any $i$.

\medskip

In \cite{AntonelliPasqualettoPozzettaSemola}, the above results are crucially exploited to show useful properties of the isoperimetric profile of $\RCD(K,N)$ spaces, without assumptions on the existence of isoperimetric sets. Building on such properties we shall prove a more explicit upper bound on the number $\overline{N}$ in \cref{thm:MassDecompositionINTRO}, which turns out to be bounded from above linearly in terms of the volume $V$. More precisely, we obtain the next corollary.

\begin{corollary}\label{cor:LinearUpperBound}
Let $K\leq 0$, $N \geq 2$, and $v_0>0$. Then there exists $\eps=\eps(K,N,v_0)>0$ such that the following holds.

Let $(X,\dist,\mathcal{H}^N)$ be a noncompact $\RCD(K,N)$ space such that $\mathcal{H}^N(B_1(x))\geq v_0$ for every $x\in X$. Let $V>0$ and let $\Omega_i\subset X$ be a minimizing (for the perimeter) sequence of bounded sets of volume $V$. Letting $\overline{N}$ be given by \cref{thm:MassDecompositionINTRO}, then
\begin{equation*}
    \overline{N} \le 1 + \frac{V}{\eps}.
\end{equation*}
\end{corollary}
\medskip

We further present the last result of this work. Employing \cref{thm:RitoreRosalesNonSmooth}, we shall give an equivalent condition for a minimizing sequence for the isoperimetric problem to converge to an isoperimetric set without losing mass at infinity, see \cref{thm:AbstractCriterionCompactness}. This abstract condition involves a new function $I^\infty_X$, which can be interpreted as the isoperimetric profile at infinity (see \cref{def:Isoprofileatinfinity}).

\begin{theorem}\label{thm:AbstractCriterionCompactness}
Let $K\leq 0$, and let $N\geq 2$. Let $(X,\dist,\mathcal{H}^N)$ be a noncompact $\RCD(K,N)$ space. Assume that the 
isoperimetric profile $I_X$ is continuous.

Let $V \in (0,\haus^N(X))$. Then the following are equivalent.
\begin{enumerate}
    \item $I_X(V) < I_X(V_1) + I^\infty_X(V_2)$ for all $V_1+V_2=V$ with $V_2 \in (0,V]$.
    
    \item For any sequence $\Omega_k \subset X$ such that $\haus^N(\Omega_k)\to V$ and $P(\Omega_k)\to I_X(V)$, there exists a subsequence converging in $L^1(X)$.
\end{enumerate}
\end{theorem}

Let us briefly comment on the statement above. 
If $K\in \R$, $N\geq 2$, and $(X,\dist,\mathcal{H}^N)$ is an $\RCD(K,N)$ space, we recall that the assumption that $I_X$ is continuous is equivalent to require that it is lower semicontinuous. Indeed, by \cite[Lemma 3.1]{AFP21}, the isoperimetric profile of these spaces is always upper semicontinuous.

Notice also that the item (1) in \cref{thm:AbstractCriterionCompactness} implies that for any sequence $\Omega_k\subset X$ such that $\mathcal{H}^N(\Omega_k)=V$ and $P(\Omega_k)\to I_X(V)$ there exists a subsequence converging in $L^1(X)$, even without assuming the continuity of the isoperimetric profile. The proof is done verbatim as in the proof of \cref{thm:AbstractCriterionCompactness}, by using \cref{thm:RitoreRosalesNonSmooth}.

Moreover, the latter implication holds on open subsets of complete Riemannian manifolds because its proof only relies on Ritoré--Rosales Theorem, cf. \cite[Theorem 2.1]{RitRosales04} (see also the discussion in \cref{rem:PerimetriAZero}).

We also observe that Item (2) (or, equivalently, (1)) in \cref{thm:AbstractCriterionCompactness} is a much stronger property than just existence of isoperimetric regions. Such property clearly does not hold on $\mathbb R^N$, or in general on noncompact spaces endowed with some strong homogeneity structure, like simply connected models of constant sectional curvature $K\leq 0$. However, in presence of a homogeneity structure on the space, one expects to be able to apply \cref{thm:MassDecompositionINTRO} much more directly and that \cref{thm:AbstractCriterionCompactness} is not needed (see also \cite{GalliRitore, NovagaStepanov21}).

On the other hand, one expects that \cref{thm:AbstractCriterionCompactness} is applicable on spaces with no homogeneity invariance satisfying some asymptotic assumptions. For instance, it follows from \cref{lem:DisugProfiliInfinito} that if $I_X$ is strictly subadditive on $(0,V]$, then item (1) is satisfied, and thus by \cref{thm:AbstractCriterionCompactness} existence of isoperimetric regions of volume $V$ holds. After \cite{AntonelliPasqualettoPozzettaSemola}, this reasoning applies, for example, to every noncompact $\RCD(0,N)$ space $(X,\dist,\haus^N)$ (different from $\R^N$) that is Gromov--Hausdorff asymptotic to $\R^N$ at infinity. This also extends to the nonsmooth setting some existence theorems in \cite{AFP21}.

\medskip

We conclude this introduction by mentioning some applications in relation with the existing literature.

Understanding existence or nonexistence of isoperimetric regions in relation with the geometric properties of an ambient space is certainly an intriguing problem in its own.
As mentioned above, existence issues in Euclidean solid cones have been investigated in \cite{RitRosales04}. In more general convex bodies, it is treated in \cite{LeonardiRitore}.
The existence of isoperimetric sets on Riemannian manifolds $(M^n, g)$ with compact quotient $M/\mathrm{Iso}(M^n)$ has been pointed out by Morgan \cite[Chapter 3]{MorganBook}, building also on \cite{AlmgrenBook}.
On nonnegatively curved surfaces, a complete positive answer to the existence of isoperimetric sets has been given in \cite{RitoreExistenceSurfaces01}. The existence of isoperimetric sets in $3$-manifolds with nonnegative scalar curvature and asymptotically flat asymptotics has been established in \cite{Carlotto2016}. Existence results for isoperimetric sets of large volumes in asymptotically flat manifolds were also obtained in \cite{EichmairMetzger, EichmairMetzger2, NardulliFloresCAG20}, in asymptotically hyperbolic spaces in \cite{Chodosh2016}, and in the asymptotically conical case in \cite{ChodoshEichmairVolkmann17}. When the ambient space is a nonnegatively Ricci curved cone, isoperimetric regions exist for any given volume and are actually characterized \cite{MorganRitore02}. Recently, the isoperimetric problem for large volumes in nonnegatively Ricci curved manifolds has been studied in \cite{AntonelliBrueFogagnoloPozzetta2021}.
As discussed more diffusely in \cite{AFP21}, it is possible to recover several of the above mentioned existence result by employing a tool like \cref{thm:MassDecompositionINTRO}.

On the other hand, existence issues for the isoperimetric problem are also related to the other important questions on the geometry of the ambient where the problem is set. For instance, the problem is related to the existence of foliations by constant mean curvature hypersurfaces in the end of a manifold \cite{Chodosh2016, ChodoshEichmairVolkmann17}. Moreover, the existence of isoperimetric regions has been crucial for the derivation of differential properties of the isoperimetric profile \cite{BavardPansu86, MorganJohnson00, Bayle03, Bayle04, BayleRosales}. As shown in \cite{AntonelliPasqualettoPozzettaSemola}, \cref{thm:MassDecompositionINTRO} is a crucial ingredient for proving such properties without assuming existence of isoperimetric regions, as well as, for deriving isoperimetric inequalities and other geometric functional inequalities like the ones contained in \cite{AgostinianiFogagnoloMazzieri, BaloghKristaly, BrendleSobolev21}.

\medskip

\textbf{Organization.} In \cref{sec:Preliminaries} we show a concentration-compactness criterion in arbitrary metric measure spaces, see \cref{lem:CoCo}, which will be useful for the proof of \cref{thm:GeneralizedCompactness}. In addition we discuss basic properties of BV functions, sets of finite perimeter, and $\CD$ and $\RCD$ spaces. We discuss the relative isoperimetric inequality, together with some consequences. After the discussion of convergence and stability properties of finite perimeter sets along converging sequences of $\RCD$ spaces, we conclude the section by proving the local H\"older continuity of the isoperimetric profile on $\RCD(K,N)$ spaces with a uniform bound from below on the volume of unit balls, see \cref{lem:ProfileHolder}.\\ In \cref{sec:RitRos} we prove the first main mass splitting result for minimizing sequences, see \cref{thm:RitoreRosalesNonSmooth}, which extends to our setting
the one in \cite{RitRosales04}.
\\ In \cref{sec:MAIN} we prove the main results of the paper \cref{thm:GeneralizedCompactness}, and \cref{thm:MassDecompositionINTRO}.\\
In \cref{sec:Equivalent} we prove \cref{thm:AbstractCriterionCompactness}.

\medskip
\textbf{Acknowledgements.} The first author is partially supported by the European Research Council (ERC Starting Grant 713998 GeoMeG `\emph{Geometry of Metric Groups}'). The second author is partially sponsored by ``Jovens Pesquisadores em Centros Emergentes" (JP-FAPESP, 21/05256-0), Brazil.

\section{Preliminaries and auxiliary results}\label{sec:Preliminaries}

We recall that a {\em metric measure space, $\mathrm{m.m.s.}$ for short,} $(X,\dist,\mathfrak{m})$ is a triple where $(X,\dist)$ is a locally compact separable metric space and $\mathfrak{m}$ is a nonnegative Borel measure bounded on bounded sets. A {\em pointed metric measure space} is a quadruple $(X,\dist,\mathfrak{m},x)$ where $(X,\dist,\mathfrak{m})$ is a metric measure space and $x\in X$ is a point. For simplicity, and since it will always be our case, we will always assume that given $(X,\dist,\meas)$ a m.m.s.\! the support ${\rm spt}\,\meas$ of the measure $\meas$ is the whole $X$. We denote with $B_r(x)$ the open ball of radius $r$ and center $x\in X$.

\subsection{Concentration-compactness on sequences of metric measure spaces}

The following version of the concentration-compactness theorem is a slight improvement of \cite[Lemma 4.3]{AFP21} (cf. also \cite[Lemma 2.1]{Nar14}) but it is stated in the broad generality of metric measure spaces.

\begin{lemma}[Concentration-compactness]\label{lem:CoCo}
Let $(X_i,\dist_i,\meas_i)$ be a sequence of metric measure spaces. Let $E_i\subset X_i$ be a sequence of bounded measurable sets such that $\lim_i \meas_i(E_i) = W \in(0,+\infty)$. Then, up to passing to a subsequence, exactly one of the following alternatives occur.
\begin{enumerate}
    \item\label{it:COCOvanishing} For any $R>0$ it holds
    \[
    \lim_i \sup_{p\in X_i} \meas_i(E_i \cap B_R(p)) =0.
    \]
    
    \item\label{it:COCOCompactness} There exists a sequence of points $p_i \in X_i$ such that for any $\eps\in(0,W/2)$ there exist $R\ge 1$, $i_\eps \in \N$ such that $|\meas_i(E_i \cap B_{R'}(p_i))-W| \leq \eps$ for any $i \ge i_\eps$, and any $R'\geq R$. Moreover, there is $I\in \N, r\ge 1$ such that $\meas_i(E_i \cap B_r(p_i)) \ge \meas_i(E_i \cap B_r(q))$ for any $q \in X$ and $\meas_i(E_i \cap B_r(p_i)) > W/2$ for any $i \ge I$.

    \item \label{it:COCODicotomia}
    There exist $w\in(0,W)$ such that for any $\eps\in(0,w/2)$ there exist a sequence of points $p_i \in X_i$ and $R\ge1$, $i_\eps \in \N$, and a sequence of open sets
    \[
    U_i = X_i \setminus \overline{B}_{R_i}(p_i) \quad\text{for some $R_i\to + \infty$},
    \]
    such that
    \begin{equation}\label{eq:Dicotomia}
        \begin{split}
            |\meas_i(E_i \cap B_R(p_i)) - w | < \eps,& \\
            |\meas_i(E_i \cap U_i) - (W-w)| < \eps, &\\
             \meas_i(E_i \cap B_R(p_i)) \ge \meas_i(E_i \cap B_R(q))& \qquad \forall\,q \in X_i,
        \end{split}
    \end{equation}
    for every $i \ge i_\eps$.
\end{enumerate}
\end{lemma}

\begin{proof}
Define $Q_i(\rho)\eqdef \sup_{p\in X_i} \meas_i(E_i\cap B_\rho(p))$. The functions $Q_i:(0,+\infty)\to \R$ are nondecreasing and uniformly bounded, since $\meas_i(E_i)\to W$. Hence the sequence $Q_i$ is uniformly bounded in $BV_{\rm loc}(0,+\infty)$ and then, up to subsequence, there exists a nondecreasing function $Q \in BV_{\rm loc}(0,+\infty)$ such that $Q_i\to Q$ in $BV_{\rm loc}$ and pointwise almost everywhere. Also, let us pointwise define $Q(\rho)\eqdef \lim_{\eta\to0^+} {\rm ess}\inf_{(\rho-\eta,\rho)} Q$, so that $Q$ is defined at every $\rho \in (0,+\infty)$. Moreover, observe that $Q(\rho) \le W$ for any $\rho>0$. Now three disjoint cases can occur, distinguishing the cases enumerated in the statement.
\begin{enumerate}
\item We have that $\lim_{\rho\to+\infty} Q(\rho)=0$, and hence $Q\equiv 0$ since it is nondecreasing. Then item 1 of the statement clearly holds.

\item We have that $\lim_{\rho\to+\infty} Q(\rho)=W$. Then there is $r\ge1$ such that $\exists \lim_{i} \sup_p \meas_i(E_i \cap B_r(p)) = Q(r) \ge \tfrac34 W$. Let $p_i\in X_i$ such that $\sup_p \meas_i(E_i \cap B_r(p))=  \meas_i(E_i \cap B_r(p_i))$ for any $i$. We claim that the sequence $p_i$ satisfies the property in item 2. Indeed, let $\eps \in (0,W/2)$ be given. Arguing as above, since $\lim_{\rho\to+\infty} Q(\rho)=W$, there is a radius $r'>0$ and a sequence $p_i'\in X$ such that $\meas_i(E_i \cap B_{r'}(p_i') ) \ge W-\eps$, and $|\meas_i(E_i)-W|\leq \varepsilon$ for any $i\ge i_\eps$. Then $\dist_i(p_i,p_i') < r + r'$, for otherwise
\[
W \xleftarrow{} \meas_i(E_i) \ge \meas_i(E_i \cap B_r(p_i)) + \meas_i(E_i \cap B_{r'}(p_i')) ,
\]
and the right hand side is $> W$ for $i$ large enough. Hence taking $R=r+2r'$ we conclude that $\meas_i ( E_i \cap B_{R'}(p_i)) \geq \meas_i ( E_i \cap B_{R}(p_i)) \ge W- \eps$ for $i\ge i_\eps$. Moreover $\meas_i ( E_i \cap B_{R'}(p_i))\leq \meas_i(E_i)\leq W+\varepsilon$, and thus we get the sought claim.

\item We have that $\lim_{\rho\to+\infty} Q(\rho)=w\in(0,W)$. Then for given $\eps\in (0,w/2)$ there is $R\ge 1$ such that
\[
w-\frac\eps8 \le Q(R) = \lim_i \sup_p \meas_i(E_i \cap B_R(p)) = \lim_i \meas_i(E_i \cap B_R(p_i)),
\]
for some $p_i \in X_i$, where in the last equality we used that
\[
\sup_p \meas_i(E_i \cap B_R(p)) = \meas_i(E_i \cap B_R(p_i)),
\]
for some $p_i$ since $E_i$ is bounded. This implies that $\meas_i(E_i \cap B_R(p_i)) \ge \meas_i(E_i \cap B_R(q))$ for any $i$ and any $q \in X_i$, and there is $i_\eps$ such that $|\meas_i(E_i \cap B_R(p_i)) - w | < \eps/4$ for $i\ge i_\eps$.

For $i\ge i_\eps$, there is an increasing sequence $\rho_j\to+\infty$ such that $Q(\rho_j)=\lim_i Q_i(\rho_j)$ and we have
\begin{equation*}
    \begin{split}
        w & = \lim_{j\to+\infty} Q(\rho_j) = \lim_j \lim_i \sup_p \meas_i(E_i \cap B_{\rho_j}(p)) \ge \limsup_j \limsup_i \meas_i(E_i \cap B_{\rho_j}(p_i)) \\
        & = \limsup_j \limsup_i \left(\meas_i(E_i \cap B_{R}(p_i)) + \meas_i(E_i \cap B_{\rho_j}(p_i) \setminus B_R(p_i))\right) \\
        & \ge w- \frac\eps4 + \limsup_j \limsup_i  \meas_i(E_i \cap B_{\rho_j}(p_i) \setminus B_R(p_i)) .
    \end{split}
\end{equation*}
Then there is $j_0$ such that for any $j\ge j_0$ we have that $\rho_j>R$ and there is $i_j$, with $i_j$ increasing to $+\infty$ as $j\to+\infty$, that satisfies
\begin{equation}\label{eq:CCstima1}
\meas_i(E_i \cap B_{\rho_j}(p_i) \setminus B_R(p_i)) <\frac\eps2  \qquad \forall\, i \ge\max\{i_\eps,i_j\}.
\end{equation}
Hence define
\[
R_i \eqdef \rho_{\max\{j \st i \ge i_j\}}.
\]
In this way $\meas_i(E_i \cap B_{R_i}(p_i) \setminus B_R(p_i))  <\eps/2 $ for any $i\ge \max\{i_\eps,i_{j_0}\}$ by \eqref{eq:CCstima1}. Defining $U_i \eqdef X_i \setminus \overline{B}_{R_i}(p_i)$ we finally get
\[
\begin{split}
W \xleftarrow{} \meas_i(E_i) &= \meas_i(E_i \cap B_R(p_i)) + \meas_i( E_i \cap  B_{R_i}(p_i) \setminus B_R(p_i)) + \meas_i(E_i \cap U_i ) \\
&\le w + \frac34\eps + \meas_i(E_i \cap U_i ),
\end{split}
\]
for $i\ge \max\{i_\eps,i_{j_0}\}$. By the first line in the above identity, recalling that $|\meas_i(E_i\cap B_R(p_i))-w|<\varepsilon/4$, we also see that $\limsup_i \meas_i(E_i \cap U_i ) \le W-w+\epsilon/4$. Hence the proof of item 3 is completed renaming $\max\{i_\varepsilon,i_{j_0}\}$ into $i_\varepsilon$ and by eventually taking a slightly bigger $i$ in order to ensure the validity of the second inequality of item 3.
\end{enumerate}
\end{proof}

\begin{remark}
The assumption on $E_i$ bounded in \cref{lem:CoCo} is not crucial. It can be removed with the drawback that the points $p_i$ are quasi-maxima for $x\mapsto \meas_i(E_i\cap B_R(x))$.
\end{remark}

\subsection{Perimeter and isoperimetry on metric measure spaces and $\RCD$ spaces}

We now discuss the basic definitions and properties of the objects involved in the following.

\begin{definition}[Perimeter and isoperimetric profile]\label{def:BVperimetro}
Let $(X,\dist,\meas)$ be a metric measure space. A function $f \in L^1(X,\meas)$ is said to belong to the space of \emph{bounded variation functions} $BV(X,\dist,\meas)$ if there is a sequence $f_i \in {\rm Lip}_{\mathrm{loc}}(X)$ such that $f_i \to f$ in $L^1(X,\meas)$ and $\limsup_i \int_X \lip f_i \de \meas < +\infty$, where $\lip u (x) \eqdef \limsup_{y\to x} \frac{|u(y)-u(x)|}{\dist(x,y)}$ is the \emph{slope} of $u$ at $x$, for any accumulation point $x\in X$, and $\lip u(x):=0$ if $x\in X$ is isolated. In such a case we define
\[
|Df|(A) \eqdef \inf\left\{\liminf_i \int_A \lip f_i \de\meas \st \text{$f_i \in {\rm Lip}_{\rm loc}(A), f_i \to f $ in $L^1(A,\meas)$} \right\},
\]
for any open set $A\subset X$.

If $E\subset X$ is a Borel set and $A\subset X$ is open, we  define the \emph{perimeter $P(E,A)$  of $E$ in $A$} by
\[
P(E,A) \eqdef \inf\left\{\liminf_i \int_A \lip u_i \de\meas \st \text{$u_i \in {\rm Lip}_{\rm loc}(A), u_i \to \chi_E $ in $L^1_{\rm loc}(A,\meas)$} \right\}.
\]
We say that $E$ has \emph{finite perimeter} if $P(E,X)<+\infty$, and we denote by $P(E)\eqdef P(E,X)$. Let us remark that the set functions $|Df|, P(E,\cdot)$ above are restrictions to open sets of Borel measures that we denote by $|Df|, |D\chi_E|$ respectively, see \cite{AmbrosioDiMarino14}, and \cite{Miranda03}. 

The \emph{isoperimetric profile of $(X,\dist,\meas)$} is
\[
I_{X}(V)\eqdef \inf \left\{ P(E) \st \text{$E\subset X$ Borel, $\meas(E)=V$} \right\},
\]
for any $V\in[0,\meas(X))$. If $E\subset X$ is Borel with $\meas(E)=V$ and $P(E)=I_X(V)$, then we say that $E$ is an \emph{isoperimetric region}.
\end{definition}

\begin{definition}[PI space]\label{def:PI}
Let $(X,\dist,\meas)$ be a metric measure space.
 We say that $\meas$ is {\em uniformly locally doubling} if for every $R>0$ there exists $C_D>0$ such that the following holds 
 $$
 \meas(B_{2r}(x))\leq C_D\meas(B_r(x)), \qquad \forall x\in X\;\forall r\leq R.
 $$

We say that a {\em weak local $(1,1)$-Poincar\'{e} inequality} holds on $(X,\dist,\meas)$ if there exists $\lambda\geq 1$ such that for every $R>0$ there exists $C_P$ such that for every pair of functions $(f,g)$ where $f\in L^1_{\mathrm{loc}}(X,\meas)$, and $g$ is an upper gradient (cf. \cite[Section 10.2]{HajlaszKoskela}) of
$f$,
the following inequality holds
$$
\fint_{B_r(x)} |f-\overline f(x)|\de\meas \leq C_Pr\fint_{B_{\lambda r}(x)} g\de\meas, 
$$
for every $x\in X$ and $r\leq R$, where $\overline f(x):=\fint_{B_r(x)}f\de\meas$.

We say that $(X,\dist,\meas)$ is a {\em PI space} when $\meas$ is uniformly locally doubling and a weak local $(1,1)$-Poincar\'{e} inequality holds on $(X,\dist,\meas)$.
\end{definition}

Let us briefly introduce the so-called $\RCD$ condition for m.m.s., and discuss some basic and useful properties of it. Since we will use part of the $\RCD$ theory just as an instrument for our purposes and since we will never use in the paper the specific definition of $\RCD$ space, we just outline the main references on the subject and we refer the interested reader to the survey of Ambrosio \cite{AmbrosioSurvey} and the references therein. In this paper we adopt the definition of $\CD$ space as in \cite[Definition 5.4]{AmbrosioSurvey} and of $\RCD$ space as in \cite[Definition 7.4]{AmbrosioSurvey}.

After the introduction, in the independent works \cite{Sturm1,Sturm2} and \cite{LottVillani}, of the curvature dimension condition $\CD(K,N)$ encoding in a synthetic way the notion of Ricci curvature bounded from below by $K$ and dimension bounded above by $N$, the definition of $\RCD(K,N)$ m.m.s.\! was first proposed in \cite{GigliRCD} and then studied in \cite{Gigli13, ErbarKuwadaSturm15,AmbrosioMondinoSavare15}, see also \cite{CavallettiMilman16} for the equivalence between the $\RCD^*(K,N)$ and the $\RCD(K,N)$ condition. The infinite dimensional counterpart of this notion had been previously investigated in \cite{AmbrosioGigliSavare14}, see also \cite{AmbrosioGigliMondinoRajala15} for the case of $\sigma$-finite reference measures.

Due to the compatibility of the $\RCD$ condition with the smooth case of Riemannian manifolds with Ricci curvature bounded from below and to its stability with respect to pointed measured Gromov--Hausdorff convergence, limits of smooth Riemannian manifolds with Ricci curvature uniformly bounded from below by $K$ and dimension uniformly bounded from above by $N$ are $\RCD(K,N)$ spaces. Then the class of $\RCD$ spaces includes the class of Ricci limit spaces, i.e., limits of sequences of Riemannian manifolds with the same dimension and with Ricci curvature uniformly bounded from below \cite{ChCo0,ChCo1,ChCo2,ChCo3}.
An extension of non collapsed Ricci limit spaces is the class of $\RCD(K,N)$ spaces where the reference measure is the $N$-dimensional Hausdorff measure relative to the distance, introduced and studied in \cite{Kitabeppu17, DePhilippisGigli18, AntBruSem19}.

\begin{definition}\label{def:Ndimensional}
Given $K\in\mathbb R$ and $N\geq 1$, we call \emph{$N$-dimensional $\RCD(K,N)$ space} any metric measure space $(X,\dist,\haus^N)$ which is an $\RCD(K,N)$ space.
\end{definition}

We stress that we adopt a different terminology with respect to \cite{DePhilippisGigli18}, in which the spaces in \cref{def:Ndimensional} are called \emph{non collapsed $\RCD$ spaces}, $\mathrm{ncRCD}$ for short. This is to avoid confusion with the fact that usually in the literature a Riemannian manifold $M$ with volume measure $\mathrm{vol}$ is said to be non collapsed when $\inf_{x\in M}\vol(B_1(x))>0$.
As a consequence of the rectifiability of $\RCD$ spaces \cite{MondinoNaber, BruePasqualettoSemolaRectifiabilitySpace2020}, we remark that if $(X,\dist,\mathcal{H}^N)$ is an $\RCD(K,N)$ space then $N$ is an integer.

\begin{remark}\label{rem:BGBella}
If $(X,\dist,\meas)$ is a metric measure space where $\meas$ is uniformly locally doubling, then for every $R>0$ there exist constants $C_1=C_1(R)$, $s:=\log_2 C_D$, where $C_D$ is the doubling constant associated to $R$, such that 
$$
\frac{\meas(B_{r_2}(x))}{\meas(B_{r_1}(x))}\leq C_1\left(\frac{r_2}{r_1}\right)^{s},
$$
for every $x\in X$ and every $r_1\leq r_2\leq R$. This is a pretty standard observation coming from the iteration of the uniformly local doubling property, cf. \cite[Lemma 14.6]{HajlaszKoskela}.

If $(X,\dist,\meas)$ is $\CD(K,N)$, then the same holds with $s=N$ and $C_1=C_1(K,N,R)$.
\end{remark}

In the following remark we record the famous Bishop--Gromov comparison results.

\begin{remark}[Models of constant sectional curvature and Bishop--Gromov Comparison Theorem]\label{rem:PerimeterMMS2}
Let $N \in \N$ with $N\ge 2$ and $K \in \R$. 

We denote by $v(N,K,r)$ the volume of the ball of radius $r$ in the (unique up to isometry) simply connected Riemannian manifold of sectional curvature $K$ and dimension $N$, and by $s(N, K, r)$ the volume of the boundary of such a ball.

For an arbitrary $\CD((N-1)K,N)$ space $(X,\dist,\meas)$ the classical Bishop--Gromov volume comparison holds. More precisely, for a fixed $x\in X$, the function $\meas(B_r(x))/v(N,K,r)$ is nonincreasing in $r$ and the function $P(B_r(x))/s(N,K,r)$ is essentially nonincreasing in $r$, i.e., the inequality
\[
P(B_R(x))/s(N,K,R) \le P(B_r(x))/s(N,K,r),
\]
holds for almost every radii $R\ge r$, see \cite[Theorem 18.8, Equation (18.8), Proof of Theorem 30.11]{VillaniBook}. Moreover, it holds that 
\begin{equation}\label{eqn:BGPer}
P(B_r(x))/s(N, K, r)\leq \meas(B_r(x))/v(N,K,r),
\end{equation}
for any $r>0$, indeed the last inequality follows from the monotonicity of the volume and perimeter ratios together with the coarea formula on balls.

In addition, if $(X,\dist,\mathcal{H}^N)$ is an $\RCD((N-1)K,N)$ space, one can conclude that $\mathcal{H}^N$-almost every point has a unique measured Gromov--Hausdorff tangent isometric to $\mathbb R^N$ (\cite[Theorem 1.12]{DePhilippisGigli18}), and thus, from the volume convergence in \cite{DePhilippisGigli18}, we get
 \begin{equation}\label{eqn:VolumeConv}
 \lim_{r\to 0}\frac{\mathcal{H}^N(B_r(x))}{v(N,K,r)}=\lim_{r\to 0}\frac{\mathcal{H}^N(B_r(x))}{\omega_Nr^N}=1, \qquad \text{for $\mathcal{H}^N$-almost every $x$},
 \end{equation}
 where $\omega_N$ is the Euclidean volume of the Euclidean unit ball in $\mathbb R^N$.
 Moreover, since the density function $x\mapsto \lim_{r\to 0}\mathcal{H}^N(B_r(x))/\omega_Nr^N$ is lower semicontinuous (\cite[Lemma 2.2]{DePhilippisGigli18}), it is bounded above by the constant $1$. Hence, from the monotonicity at the beginning of the remark we deduce that, if $(X,\dist,\mathcal{H}^N)$ is an $\RCD((N-1)K,N)$ space, then for every $x\in X$ we have 
 \begin{equation}\label{eqn:BGVol}
 \mathcal{H}^N(B_r(x))\leq v(N,K,r),
 \end{equation}
 for every $r>0$. Furthermore, combining \eqref{eqn:BGVol} with \eqref{eqn:BGPer}, we also have that $P(B_r(x))\le s(N,K,r)$ for any $x \in X$ and $r>0$. If also $\haus^N(B_1(x))\ge v_0>0$ for any $x \in X$, then we have that if $\haus^N(B_{r_i}(x_i))\to0$ then $r_i\to0$ and $P(B_{r_i}(x_i))\to0$.
\end{remark}

We state a general covering lemma proved in \cite{AntonelliPasqualettoPozzetta}.

\begin{lemma}[{Covering Lemma, \cite[Lemma 3.7]{AntonelliPasqualettoPozzetta}}]\label{lem:CoveringLemma}
Let $(X,\dist,\meas)$ be a metric measure space where $\meas$ is uniformly locally doubling, let $\overline R>0$, and let $\rho\leq \overline R$.
\begin{enumerate}
    \item Let $C_1,s$ be the constants associated to the radius $\overline R$ as in \cref{rem:BGBella}. Hence, for any $\alpha>0$, $z \in X$, $\alpha \rho \le R \le \overline R$, it holds
    \begin{equation*}
        \sharp \mathscr{F}  \le C_1\left(\frac{2R}{\alpha\rho}\right)^{s},
    \end{equation*}
    for any family $\mathscr{F}$ of disjoint balls of radius $\alpha\rho$ contained in $B_R(z)$.
    
    \item If $\Omega \subset X$ is open and $D\subset \Omega$ is dense in $\Omega$, there exist countably many points $\{x_i\}_{i\in\N} \subset D$ such that
    \begin{equation}\label{eq:Covering}
    \begin{split}
        B_{\frac\rho2}(x_i) \cap B_{\frac\rho2}(x_j) = \emptyset& \qquad \forall\,i\neq j, \\
        \bigcup_i B_\rho(x_i) \supset D& , \\
        \bigcup_i B_{\lambda\rho}(x_i) \supset \Omega&
        \qquad \forall\,\lambda>1.
    \end{split}
    \end{equation}
    Moreover, for any $\beta\leq \overline R$ and $z \in \bigcup_i B_{\beta\rho}(x_i)$, it holds
    \begin{equation}\label{eq:StimaCoveringOverlap}
        \sharp \left\{ \text{balls $B_{\beta\rho}(x_i) \st z \in B_{\beta\rho}(x_i)$} \right\} \le \max\left\{ 1,C_1(8\beta)^{s} \right\},
    \end{equation}
    where $C_1,s$ are the constants as in \cref{rem:BGBella} associated to $2\overline R^2$. 
\end{enumerate}
If $(X,\dist,\meas)$ is $\CD(K,N)$, then the same holds with $s=N$.
\end{lemma}

\begin{remark}[Relative isoperimetric inequalities]\label{rem:RelativeIsoperimetric} Relative isoperimetric inequalities are today well understood in the following settings of metric measure spaces.

\begin{itemize}
    
    \item \textsc{PI spaces.} Let $(X,\dist,\meas)$ be a PI space. Let $\Omega\subset X$ be a bounded set, let $R>0$, and let $C_D, C_P$ be the corresponding doubling constant and Poincar\'{e} constant. Defining $s\eqdef \log_2 C_D$, there exists a constant $C_{\rm RI}'=C_{\rm RI}'(C_D,C_P, \lambda, R, \Omega)$ such that
    \begin{equation}\label{eq:RelativeIsopPI}
    \min\left\{ \meas(B_r(x)\cap E)^{1-\frac1s} , \meas(B_r(x)\setminus E)^{1-\frac1s}  \right\} \le C_{\rm RI}' P(E, B_{5\lambda r}(x)),
    \end{equation}
    for any $r\le R$ and $x \in \Omega$, for any set of locally finite perimeter $E$ in $X$.

    Moreover, if there exists $v_0>0$ such that $\meas(B_1(x))\geq v_0$ for every $x\in X$, then the constant $C_{\mathrm{RI}}'$ only depends on $C_D,C_P, \lambda,R,v_0$.

    \item \textsc{$\CD(K,N)$ spaces.} Let $(X,\dist,\meas)$ be an $\CD(K,N)$ space. Let $\Omega\subset X$ be a bounded set and let $R>0$. Then there is a constant $C_{\rm RI}= C_{\rm RI}(N,K,\Omega,R)$ such that
    \begin{equation}\label{eq:RelativeIsopRCD}
    \min\left\{ \meas(B_r(x)\cap E)^{1-\frac1N} , \meas(B_r(x)\setminus E)^{1-\frac1N}  \right\} \le C_{\rm RI} P(E, B_r(x)),
    \end{equation}
    for any $r\le R$ and $x \in \Omega$, for any set of locally finite perimeter $E$ in $X$. 

    Moreover, if there exists $v_0>0$ such that $\meas(B_1(x))\geq v_0$ for every $x\in X$, then the constant $C_{\mathrm{RI}}$ only depends on $N,K,R,v_0$.
\end{itemize}
\end{remark}

The next lemma essentially states that, in the suitable setting, it is always possible to find a ball where the mass of a set of finite perimeter concentrates. Such result will be properly used to rule out the occurrence of the first alternative in \cref{lem:CoCo}.

\begin{lemma}[Local mass lower bound]\label{lem:MasLowBound}
Let $(X,\dist,\meas)$ be a $\CD(K,N)$ space. Assume that $\meas(B_1(x))\ge v_0>0$ for any $x\in X$. Then there exists a constant $C_M=C_M(N,K,v_0)>0$ such that for any finite perimeter set $E$ with $\meas(E)\in(0,+\infty)$ there exists $x_0\in X$ such that
\[
\meas(E\cap B_1(x_0) ) \ge \min \left\{C_M \frac{\meas(E)^N}{P(E)^N}, \frac{v_0}{2} \right\}.
\]
\end{lemma}

\begin{proof}
If there is $x_0 \in X$ such that $\meas(E \cap B_1(x_0)) \ge \tfrac12 \meas(B_1(x_0))$, then $\meas(E \cap B_1(x_0)) \ge v_0/2$. So we can assume instead that
\begin{equation}\label{eq:MLB1}
\meas(E \cap B_1(x)) < \frac12 \meas(B_1(x)) 
\qquad
\forall\, x\in X.
\end{equation}
We apply item (2) in \cref{lem:CoveringLemma} with $\Omega=D=X$ and $\beta= \rho =1$, which yields a covering $\{ B_1(x_i)\}_{i \in \N}$ of $X$. Let $i_0$ be such that
\[
L \eqdef \sup_{i\in\N} \meas(E \cap B_1(x_i))^{\frac1N} \le 2 \meas(E \cap B_1(x_{i_0}))^{\frac1N}.
\]
By \eqref{eq:MLB1} and \eqref{eq:RelativeIsopRCD} we get
\begin{equation}\label{eq:MLB2}
    \meas(E \cap B_1(x))^{\frac{N-1}{N}} \le C_{\rm RI} P(E,B_1(x))
    \qquad
\forall\, x \in X,
\end{equation}
where $C_{\rm RI}$ is the constant in \eqref{eq:RelativeIsopRCD} associated to $R=1$.
Therefore, using \eqref{eq:MLB2}, \eqref{eq:StimaCoveringOverlap}, and the choice of $i_0$, we estimate
\begin{equation}\label{eq:MLB3}
    \begin{split}
        \meas (E)
        & \le \sum_i \meas(E \cap B_1(x_i)) = \sum_i \meas(E \cap B_1(x_i))^{\frac1n} \meas(E \cap B_1(x_i))^{\frac{N-1}{N}} \\
        &\le L \sum_i  C_{\rm RI} P(E,B_1(x_i)) \le 2 C_{\rm RI}\max\{1, C_18^N\} P(E)  \meas(E \cap B_1(x_{i_0}))^{\frac1N},
    \end{split}
\end{equation}
where $C_1$ is as in \cref{rem:BGBella} with $R=1$ and depends on $K,N$ only. Hence \eqref{eq:MLB3} completes the proof.
\end{proof}

\begin{remark}\label{rem:ScaledMassLowBound}
\cref{lem:MasLowBound} clearly holds at any fixed scale. More precisely, under the same assumptions of \cref{lem:MasLowBound}, for any $R>0$ there exist constants $C_{1}(R,N,K,v_0)>0$ and $C_{2}(R,N,K)>0$ such that for any finite perimeter set $E$ with $\meas(E)\in(0,+\infty)$ for any $r \in (0,R]$ there exists $x_0\in X$ such that
\[
\meas(E\cap B_r(x_0) ) \ge \min \left\{C_{1}(R,N,K,v_0) \frac{\meas(E)^N}{P(E)^N}, C_{2}(R,N,K)v_0 r^N \right\}.
\]
\end{remark}

In the following treatment we introduce the pmGH-convergence already in a proper realization even if this is not the general definition. Nevertheless, the (simplified) definition of Gromov--Hausdorff convergence via a realization is equivalent to the standard definition of pmGH convergence in our setting, because in the applications we will always deal with locally uniformly doubling measures, see \cite[Theorem 3.15 and Section 3.5]{GigliMondinoSavare15}. The following definition is actually taken from the introductory exposition of \cite{AmborsioBrueSemola19}.

\begin{definition}[pGH and pmGH convergence]\label{def:GHconvergence}
A sequence $\{ (X_i, \dist_i, x_i) \}_{i\in \N}$ of pointed metric spaces is said to converge in the \emph{pointed Gromov--Hausdorff topology, in the $\mathrm{pGH}$ sense for short,} to a pointed metric space $ (Y, \dist_Y, y)$ if there exist a complete separable metric space $(Z, \dist_Z)$ and isometric embeddings
\[
\begin{split}
&\Psi_i:(X_i, \dist_i) \to (Z,\dist_Z), \qquad \forall\, i\in \N,\\
&\Psi:(Y, \dist_Y) \to (Z,\dist_Z),
\end{split}
\]
such that for any $\eps,R>0$ there is $i_0(\varepsilon,R)\in\mathbb N$ such that
\[
\Psi_i(B_R^{X_i}(x_i)) \subset \left[ \Psi(B_R^Y(y))\right]_\eps,
\qquad
\Psi(B_R^{Y}(y)) \subset \left[ \Psi_i(B_R^{X_i}(x_i))\right]_\eps,
\]
for any $i\ge i_0$, where $[A]_\eps\eqdef \{ z\in Z \st \dist_Z(z,A)\leq \eps\}$ for any $A \subset Z$.

Let $\meas_i$ and $\mu$ be given in such a way $(X_i,\dist_i,\meas_i,x_i)$ and $(Y,\dist_Y,\mu,y)$ are m.m.s.\! If in addition to the previous requirements we also have $(\Psi_i)_\sharp\mathfrak{m}_i \rightharpoonup \Psi_\sharp \mu$ with respect to duality with continuous bounded functions on $Z$ with bounded support, then the convergence is said to hold in the \emph{pointed measure Gromov--Hausdorff topology, or in the $\mathrm{pmGH}$ sense for short}.
\end{definition}

\begin{remark}[pmGH limit of $\RCD$ spaces]
	We recall that, whenever it exists, a pmGH limit of a sequence $\{(X_i,\dist_i,\meas_i,x_i)\}_{i\in\mathbb N}$ of (pointed) $\RCD(K,N)$ spaces is still an $\RCD(K,N)$ metric measure space.
\end{remark}

\begin{remark}[Gromov precompactness theorem for $\RCD$ spaces]\label{rem:GromovPrecompactness}
Here we recall the synthetic variant of Gromov's precompactness theorem for $\RCD$ spaces, see \cite[Equation (2.1)]{DePhilippisGigli18}. Let $\{(X_i,\dist_i,\meas_i,x_i)\}_{i\in\mathbb N}$ be a sequence of $\RCD(K_i,N)$ spaces with $N\in[1,+\infty)$, $\supp(\meas_i)=X_i$ for every $i\in\mathbb N$, $\meas_i(B_1(x_i))\in[v,v^{-1}]$ for some $v\in(0,1)$ and for every $i\in\mathbb N$, and $K_i\to K\in\mathbb R$. Then there exists a subsequence pmGH-converging to some $\RCD(K,N)$ space $(X,\dist,\meas,x)$ with $\supp(\meas)=X$.
\end{remark}

We need to recall a generalized $L^1$-notion of convergence for sets defined on a sequence of metric measure spaces converging in the pmGH sense. Such a definition is given in \cite[Definition 3.1]{AmborsioBrueSemola19}, and it is investigated in \cite{AmborsioBrueSemola19} capitalizing on the results in \cite{AmbrosioHonda17}.

\begin{definition}[$L^1$-strong and $L^1_{\mathrm{loc}}$ convergence]\label{def:L1strong}
Let $\{ (X_i, \dist_i, \mathfrak{m}_i, x_i) \}_{i\in \N}$  be a sequence of pointed metric measure spaces converging in the pmGH sense to a pointed metric measure space $ (Y, \dist_Y, \mu, y)$ and let $(Z,\dist_Z)$ be a realization as in \cref{def:GHconvergence}.

We say that a sequence of Borel sets $E_i\subset X_i$ such that $\mathfrak{m}_i(E_i) < +\infty$ for any $i \in \N$ converges \emph{in the $L^1$-strong sense} to a Borel set $F\subset Y$ with $\mu(F) < +\infty$ if $\mathfrak{m}_i(E_i) \to \mu(F)$ and $\chi_{E_i}\mathfrak{m}_i \rightharpoonup \chi_F\mu$ with respect to the duality with continuous bounded functions with bounded support on $Z$.

We say that a sequence of Borel sets $E_i\subset X_i$ converges \emph{in the $L^1_{\mathrm{loc}}$-sense} to a Borel set $F\subset Y$ if $E_i\cap B_R(x_i)$ converges to $F\cap B_R(y)$ in $L^1$-strong for every $R>0$.
\end{definition}

The next statements collect precompactness, lower semicontinuity, and strong approximation results in the setting of pmGH converging ambient spaces.

\begin{proposition}[{\cite[Proposition 3.3, Corollary 3.4, Proposition 3.6, Proposition 3.8]{AmborsioBrueSemola19}}]\label{prop:SemicontinuitaAmbrosioBrueSemola}
Let $K\in\mathbb R$, $N\geq 1$, and $\{(X_i,\dist_i,\meas_i,x_i)\}_{i\in\mathbb N}$ be a sequence of $\RCD(K,N)$ m.m.s.\! converging in the pmGH sense to $(Y,\dist_Y,\mu,y_0)$. Then the following hold:
\begin{itemize}
    \item For any $r>0$ and for any sequence of finite perimeter sets $E_i\subset \overline B_{r}(x_i)$  satisfying
    $$
    \sup_{i\in\mathbb N}|D\chi_{E_i}|(X_i)<+\infty,
    $$
    there exists a subsequence $i_k$ and a finite perimeter set $F\subset \overline{B}_r(y_0)$ such that $E_{i_k}\to F$ in $L^1$-strong as $k\to+\infty$. Moreover 
    $$
    |D\chi_F|(Y)\leq \liminf_{k\to+\infty}|D\chi_{E_{i_k}}|(X_{i_k}).
    $$
    
    \item For any sequence of Borel sets $E_i\subset X_i$ with 
    $$
    \sup_{i\in\mathbb N}|D\chi_{E_i}|(B_R(x_i))<+\infty, \qquad \forall\,R>0,
    $$
    there exists a subsequence $i_k$ and a Borel set $F\subset Y$ such that $E_{i_k}\to F$ in $L^1_{\mathrm{loc}}$.
    
    \item Let $F\subset Y$ be a bounded set of finite perimeter. Then there exist a subsequence $i_k$, and uniformly bounded finite perimeter sets $E_{i_k}\subset X_{i_k}$ such that $E_{i_k}\to F$ in $L^1$-strong and $|D\chi_{E_{i_k}}|(X_{i_k})\to |D\chi_F|(Y)$ as $k\to+\infty$.
\end{itemize}
\end{proposition}

In the case of $\RCD$ space with Hausdorff measure and lower bounds on the volume of unit balls, employing \cref{lem:MasLowBound} we can improve the last item in \cref{prop:SemicontinuitaAmbrosioBrueSemola} requiring approximation with constant mass.

\begin{lemma}\label{lem:ApproxInMassa}
Let $K\in\mathbb R$, $N\geq 1$, and $\{(X_i,\dist_i,\haus^N,x_i)\}_{i\in\mathbb N}$ be a sequence of $\RCD(K,N)$ m.m.s.\! converging in the pmGH sense to $(Y,\dist_Y,\haus^N,y_0)$. Assume that $\haus^N(B_1(y))\ge v_0>0$ for any $y \in Y$.

Let $F\subset Y$ be a bounded set of finite perimeter. Then there exist a subsequence $i_k$ and uniformly bounded finite perimeter sets $F_{i_k}\subset X_{i_k}$ such that
\begin{equation}\label{eq:ApproxInMassa}
    \begin{split}
        F_{i_k}\to F  \quad \text{$L^1$-strong,} \qquad
        P(F_{i_k}) \to P(F), \qquad
        \haus^N(F_{i_k}) = \haus^N(F) \quad\forall k.
    \end{split}
\end{equation}
\end{lemma}

\begin{proof}
We can assume that $P(F)>0$ without loss of generality. By the last item in \cref{prop:SemicontinuitaAmbrosioBrueSemola} there exist a subsequence $i_k$, and uniformly bounded finite perimeter sets $E_{i_k}\subset X_{i_k}$ such that $E_{i_k}\to F$ in $L^1$-strong and $P(E_{i_k}) \to P(F)$ as $k\to+\infty$.
By pmGH convergence we know that $\haus^N(B_1(x))\ge v_0$ for any $x \in X_{i_k}$ and any $k$. By \cref{lem:MasLowBound} there exist a constant $C_M=C_M(N,K,v_0)>0$ and a sequence of points $z_{i,k} \in X_{i_k}$ such that
\[
\haus^N(E_{i_k}\cap B_1(z_{i_k}) ) \ge \min \left\{C_M \frac{\haus^N(E_{i_k})^N}{P(E_{i_k})^N}, \frac{v_0}{2} \right\},
\]
and then $\haus^N(E_{i_k}\cap B_1(z_{i_k}) ) \ge C_0=C_0(C_M,F,v_0)>0$ for any $k$ sufficiently large.
On the other hand, by the uniform boundedness of $E_{i_k}$, we can also find a sequence of points $w_{i_k} \in X_{i_k}$ such that $\haus^N(E_{i_k}\cap B_1(w_{i_k}) )=0$ for any $k$.

Therefore, since $\haus^N(E_{i_k})\to \haus^N(F)$, there exist radii $r_{i_k,1}, r_{i_k,2} \in [0,1)$ such that either
\[
E_{i_k}\setminus B_{r_{i_k,1}}(z_{i_k})
\qquad
\text{or}
\qquad
E_{i_k}\cup B_{r_{i_k,2}}(w_{i_k})
\]
has measure equal to $\haus^N(F)$, for any $k$ large. If $r_{i_k,2}>0$ for infinitely many $k$'s, then, up to passing to a subsequence, we can take $F_{i_k}= E_{i_k}\cup B_{r_{i_k,2}}(w_{i_k})$. Indeed, since $\haus^N(E_{i_k})\to \haus^N(F)$, then $r_{i_k,2}\to0$ and $P(B_{r_{i_k,2}}(w_{i_k})\to0$, implying \eqref{eq:ApproxInMassa}.

We can argue in the very same way if $r_{i_k,2}=0$ definitely and $r_{i_k,1}>0$ for infinitely many $k$'s with $\liminf_k r_{i_k,1}=0$, taking $F_{i_k}= E_{i_k}\setminus B_{r_{i_k,1}}(z_{i_k})$ this time.

Finally, if $r_{i_k,2}=0$ definitely, $r_{i_k,1}>0$ for infinitely many $k$'s and $\liminf_k r_{i_k,1}\ge r_0>0$, then $\haus^N(E_{i_k} \cap B_{r_{i_k,1}}(z_{i_k}))\to0$ and we take $F_{i_k}= E_{i_k}\setminus B_{r_{i_k,1}}(z_{i_k})$, estimating
\[
P(F) \le \liminf_k P( E_{i_k}\setminus B_{r_{i_k,1}}(z_{i_k})) \le \limsup_k \frac{C_{K,N}}{r_0}\haus^N(E_{i_k} \cap B_{r_{i_k,1}}(z_{i_k})) + P(E_{i_k}) = P(F),
\]
where we estimated $P( E_{i_k}\setminus B_{r_{i_k,1}}(z_{i_k}))$ employing the Deformation Lemma given by \cite[Theorem 1.1]{AntonelliPasqualettoPozzetta}.
\end{proof}

The following results are shown in \cite{AFP21}.

\begin{lemma}\label{lem:ProfileOnBoundedSets}
Let $(X,\dist,\mathcal{H}^N)$ be an $\RCD (K,N)$ space with $\mathcal{H}^N(X)=+\infty$. If, for some $v_0>0$,  $\mathcal{H}^N(B_1(x))\ge v_0$ for any $x\in X$, then the isoperimetric profile $I_X$ of $X$ can be rewritten as
\begin{equation}\label{eq:ProfileBoundedSets}
  I_X(V) = \inf \left\{P(E) \st \text{$E\subset X$ Borel, $\mathcal{H}^N(E)=V$, $E$ bounded} \right\}
\qquad
\forall\, V\in (0,+\infty).  
\end{equation}
\end{lemma}

The following proposition has to be read as a generalization of \cite[Lemma 2.7]{Nar14} and \cite[Proposition 3.2]{AFP21}. We omit the proof that can be reached arguing verbatim as in \cite[Proposition 3.2]{AFP21}.

\begin{proposition}\label{prop:ComparisonIsoperimetricProfile} 
Let $(X,\dist,\mathcal{H}^N)$ be an $\RCD(K,N)$ space with $K\leq 0$ such that $\mathcal{H}^N(B_1(x))\geq v_0>0$ for every $x\in X$. Let $p_i\in X$ be a diverging sequence of points on $X$. Then, up to subsequence, there exists $(X_{\infty},\dist_{\infty},\mathcal{H}^N,p_{\infty})$ a pointed $\RCD(K,N)$ space such that
\begin{equation}\label{eq:ConvergenceIsop}
(X,\dist,\mathcal{H}^N,p_i)\xrightarrow[i\to+\infty]{pmGH}(X_{\infty},\dist_{\infty},\mathcal{H}^N,p_{\infty}).
\end{equation}

Moreover, whenever a diverging sequence of points $p_i\in X$ and a pointed $\RCD(K,N)$ space $(X_{\infty},\dist_{\infty},\mathcal{H}^N,p_{\infty})$ satisfy \eqref{eq:ConvergenceIsop}, then
\begin{equation}\label{eq:GeneralInequalityIsopProfile}
    I_{X}(V) \le I_{X}(V_1) + I_{X_\infty}(V_2) 
    \qquad
    \forall\, V= V_1+ V_2,
\end{equation}
with $V,V_1,V_2\ge 0$.
In particular
\begin{equation}\label{eq:InequalityIsopProfile}
 I_{X}(V)\leq I_{X_{\infty}}(V) \qquad \forall\, V>0,
\end{equation}
Moreover if, for any $j\ge 1$, $\{p_{i,j} \,|\, i \in \N\}$ is a diverging sequence of points on $X$ such that $(X,\dist,\mathcal{H}^N,p_{i,j})\to (X_j,\dist_j,\mathcal{H}^N,p_j)$ in the pmGH sense as $i\to+\infty$, then
\begin{equation}\label{eq:InfinityInequalityIsopProfile}
    I_{X}(V) \le I_{X}(V_0) + \sum_{j=1}^{+\infty} I_{X_j}(V_j),
\end{equation}
whenever $V= \sum_{j=0}^{+\infty} V_j$ with $V,V_j\ge 0$ for any $j$.
\end{proposition}

We will need the following recent result about the topological regularity of isoperimetric sets.

\begin{theorem}[{\cite[Theorem 1.4]{AntonelliPasqualettoPozzetta}}]\label{thm:RegularityIsoperimetricSets}
    Let $(X,\dist,\haus^N)$ be an $\RCD(K,N)$ space with $2\leq N<+\infty$ natural number, $K\in\mathbb R$, and assume there exists $v_0>0$ such that $\mathcal{H}^N(B_1(x))\geq v_0$ for every $x\in X$. Let $E\subset X$ be an isoperimetric region.
    
    Then $E^{(1)}$ is open and bounded, $\partial^eE=\partial E^{(1)}$, and $\partial E^{(1)}$ is $(N-1)$-Ahlfors regular in $X$. 
\end{theorem}

We shall employ the following formulation of deformation lemma for sets of finite perimeter in $\RCD$ spaces.

\begin{theorem}[{Measure prescribing localized deformations \cite[Theorem 2.35]{AntonelliPasqualettoPozzetta}}]\label{thm:VariazioniMaggi}
    Let \((X,\sfd,\mm)\) be an \({\sf RCD}(K,N)\) space with \(N<\infty\).
	Let \(E\subset X\) be a set of locally finite perimeter and let $A\subset X$ be a connected open set. Assume that $P(E,A)>0$.
	\begin{itemize}
	    \item[i)] If $E\cap A$ has interior points, then there exist a ball $B\Subset A$, $\eta_1=\eta_1(E,A)>0$ and $C_1(E,A)>0$ such that for every $\eta \in [0,\eta_1)$ there is a set $F\supset E$ such that
	    \[
	    E\Delta F \subset B,
	    \qquad
	    \meas(F\cap B)=\meas(E\cap B)+\eta,
	    \qquad
	    P(F,A)\le C_1(E,A)\eta+ P(E,A). 
	    \]
	    
	    \item[ii)] If $E\cap A$ has exterior points, then there exist a ball $B\Subset A$, $\eta_2=\eta_2(E,A)>0$, and $C_2(E,A)>0$ such that for every $\eta \in [0,\eta_2)$ there is a set $F\subset E$ such that
	    \[
	    E\Delta F \subset B,
	    \qquad
	    \meas(F\cap B)=\meas(E\cap B)-\eta,
	    \qquad
	    P(F,A)\le C_2(E,A)\eta+ P(E,A).
	    \]
	\end{itemize}
\end{theorem}

We recall the following isoperimetric inequality for small volumes.

\begin{proposition}[{\cite[Proposition 3.20, Remark 3.21]{AntonelliPasqualettoPozzetta}}]\label{prop:IsopVolumiPiccoli}
Let $(X,\dist,\meas)$ be a $\CD(K,N)$ space. Assume that $\meas(B_1(x))\geq v_0>0$ for every $x\in X$. Then there exist $v:=v(K,N,v_0)>0$ and $C:=C(K,N,v_0)>0$ such that for every set of finite perimeter $E$ it holds
$$
\meas(E)\leq v \quad\Rightarrow \quad \meas(E)^{N/(N-1)}\leq CP(E).
$$
\end{proposition}

The following \cref{lem:ProfileHolder} states that in the nonsmooth setting of $\RCD(K,N)$ spaces with reference measure $\mathcal{H}^N$ and volume of unit balls uniformly bounded below, the isoperimetric profile is locally $(1-1/N)$-H\"older continuous. This result gives a mild regularity of the isoperimetric profile that will be sufficient for the purposes of this work. We stress that, by using refined tools of geometric analysis in nonsmooth spaces, and building on the regularity result of \cref{lem:ProfileHolder}, much more can be said on the isoperimetric profile function and we shall address this problem in the paper \cite{AntonelliPasqualettoPozzettaSemola}. On the other hand the proof of the next result, which is  adapted from \cite[Theorem 2]{FloresNardulli20}, is based on elementary comparison arguments.

\begin{lemma}\label{lem:ProfileHolder}
Let $(X,\dist,\mathcal{H}^N)$ be an $\RCD (K,N)$ space such that $\mathcal{H}^N(B_1(x))\geq v_0>0$ for every $x \in X$. Then its isoperimetric profile $I_X:(0,\mathcal{H}^N(X))\to [0,+\infty)$ is locally $(1-1/N)$-H\"older continuous.
\end{lemma}

\begin{proof}
Let us fix $V_0\in (0,\haus^N(X))$, and let us show that $I_X$ is $(1-1/N)$-H\"older continuous locally around $V_0$. Let us assume first that $X$ is noncompact.

Let us fix $0<\varepsilon\leq 1$. From  \cref{lem:ProfileOnBoundedSets} we get that there exists a bounded set of finite perimeter $E\subset X$, with $\mathcal{H}^N(E)=V_0$, such that $P(E)\leq I_X(V_0)+\varepsilon$. Since $E$ is bounded ad $X$ is noncompact we get that there exists $B_1(x_0)\Subset X\setminus \overline E$. Since $\mathcal{H}^N(B_1(x_0))\geq v_0$, we have that for every $V\in [V_0,V_0+v_0)$ there exists an $r_V\leq 1$ such that $\mathcal{H}^N(E\sqcup B_{r_V}(x_0))=V$. Hence, for some constant $C:=C(N,K)>0$,
\begin{equation}\label{eqn:HolderAbove}
    I_X(V)\leq P(E\sqcup B_{r_V}(x_0)) = P(E)+P(B_{r_V}(x_0)) \leq P(E)+Cr_V^{N-1}\leq I_X(V_0)+Cr_V^{N-1}+\varepsilon,
\end{equation}
where we used \eqref{eqn:BGPer} and \eqref{eqn:BGVol}. By Bishop--Gromov comparison, see \cref{rem:PerimeterMMS2}, we have that 
\[
\frac{\mathcal{H}^N(B_{r_V}(x_0))}{v(N,K,r_v)}\geq \frac{\mathcal{H}^N(B_{1}(x_0))}{v(N,K,1)}=:\eta_0(N,K,v_0),
\]
and hence, since $r_V\leq 1$, there exists $\eta_1:=\eta_1(N,K,v_0)$ such that $\mathcal{H}^N(B_{r_V}(x_0))\geq \eta_1r_V^{N}$. Since $\mathcal{H}^N(B_{r_V}(x_0))=V-V_0$ we finally get that there exists $\eta_2:=\eta_2(N,K,v_0)$ such that
\[
r_V^{N-1}\leq \eta_2(V-V_0)^{\frac{N-1}{N}},
\]
and hence, by inserting the last inequality in \eqref{eqn:HolderAbove} we get that there exists $\eta_3:=\eta_3(N,K,v_0)$ such that for every $V\in [V_0,V_0+v_0)$ we have
\begin{equation}\label{eqn:Control1}
I_X(V)\leq I_X(V_0)+\eta_3(V-V_0)^{\frac{N-1}{N}}+\varepsilon.
\end{equation}

Let us now deal with the case in which $V\leq V_0$. By \cref{rem:ScaledMassLowBound} we get that there exists a constant $C':=C'(N,K,v_0)$ and a point $z\in X$ such that for every $0<r\leq 1$, the following inequality holds
\begin{equation}\label{eqn:ScaledMassLowerBound}
\mathcal{H}^N(E\cap B_r(z))\geq C'\min\left\{\frac{\left(\mathcal{H}^N(E)\right)^N}{P(E)^N},v_0r^N\right\}\geq C'\min\left\{\frac{V_0^N}{\left(I_X(V_0)+\varepsilon\right)^N},v_0r^N\right\},
\end{equation}
where in the second inequality holds by the choice of $E$.
Let us define 
\begin{equation}\label{eqn:DefinitionV1}
    v_1:=C'\min\left\{\frac{V_0^N}{(I_X(V_0)+\varepsilon)^N},v_0\right\},
\end{equation}
and let us take an arbitrary $V\in (\max\{0,V_0-v_1\},V_0]$. Since 
\begin{equation}\label{eqn:tV}
t_V:=\left(\frac{V_0-V}{C'v_0}\right)^{1/N}\leq \min\left\{1,\frac{V_0}{v_0^{1/N}(I_X(V_0)+\varepsilon)}\right\}\leq 1
\end{equation}
by the very definition of $v_1,V$, we can use \eqref{eqn:ScaledMassLowerBound} to obtain
\[
\mathcal{H}^N(E\cap B_{t_V}(z))\geq C'\min\left\{\frac{V_0^N}{\left(I_X(V)+\varepsilon\right)^N},v_0t_V^N\right\}=V_0-V,
\]
where in the last equality we exploited \eqref{eqn:tV}. Hence, from the previous inequality, we get that there exists $t_V'\leq t_V$ such that 
\[
\mathcal{H}^N(E\setminus B_{t_V'}(z))=V.
\]
Hence, for some constants  $C:=C(N,K)$, and $\widetilde C:=\widetilde C(N,K,v_0)$, the following inequality holds
\begin{equation}\label{eqn:Control2}
\begin{split}
I_X(V)&\leq P(E\setminus B_{t_V'}(z)) \leq P(E)+P(B_{t_v'}(z))\leq I_X(V_0)+\varepsilon+Ct_V'^{N-1} \\
&\leq I_X(V_0)+\varepsilon+\widetilde C(V_0-V)^{\frac{N-1}{N}}.
\end{split}
\end{equation}
where in the third inequality we exploited the choice of $E$ and Bishop--Gromov comparison (see \eqref{eqn:BGVol}, \eqref{eqn:BGPer}), and in the fourth inequality we used $t_V'\leq t_V$ and the definition of $t_V$.

Summing up, we proved the following. There exists a constant $\vartheta:=\vartheta(N,K,v_0)$ such that the following holds. Given $V_0\in (0,+\infty)$ and any $\varepsilon\in(0,1)$, letting $v_1$ be as in $\eqref{eqn:DefinitionV1}$, for every $V\in J\eqdef  (\max\{0,V_0-v_1\},V_0+v_0)$ we have 
\begin{equation}\label{eqn:INEQFINAL}
I_X(V)\leq I_X(V_0)+\vartheta|V-V_0|^{\frac{N-1}{N}}+\varepsilon.
\end{equation}

By using the previous inequality, up to shrinking the neighborhood $J$, we get that for every $V\in J$ the quantity
\[
C'\min\left\{\frac{V^N}{(I_X(V)+\varepsilon)^N},v_0\right\},
\]
is uniformly bounded below by some number $\gamma>0$, independently of $\varepsilon\in(0,1)$. Hence, up to shrinking the neighborhood $J$ in such a way that its length is less than $\min\{\gamma,v_0\}$, from \eqref{eqn:INEQFINAL}, for every $V,W\in J$ and for every $\varepsilon\in(0,1)$ the following inequality holds
\[
I_X(V)\leq I_X(W)+\vartheta|V-W|^{\frac{N-1}{N}}+\varepsilon.
\]
Taking $\varepsilon\to 0$ we get the sought conclusion.

It remains to consider the case in which $X$ is compact. In this case the classical direct method of Calculus of Variations implies that for every volume $V\in (0,\mathcal{H}^N(X))$ there exists an isoperimetric region $E$ of volume $V$ in $X$. Hence simpler comparison arguments can be performed exploiting the regularity result in \cref{thm:RegularityIsoperimetricSets}, which implies that $E$ has an interior (resp., exterior) point $x_0$ (resp., $x_1$). Thus one can reproduce the arguments above, which become slightly simpler, by adding (resp., subtracting) sufficiently small balls around $x_0$ (resp., $x_1$), using Bishop--Gromov comparison, and estimates from a deformation lemma like \cite[Theorem 1.1]{AntonelliPasqualettoPozzetta}.
\end{proof}


\section{Mass splitting of perimeter minimizing sequences on $\RCD$ spaces}\label{sec:RitRos}

In this section we identify a general phenomenon of perimeter minimizing sequences of sets, whose mass splits in a converging component and in a diverging one. This is the first step required for the proof of \cref{thm:MassDecompositionINTRO} and it can be seen as a generalization of \cite[Theorem 2.1]{RitRosales04}.

Let us first recall the well-understood property of $BV$ precompactness in the general framework of $\CD$ spaces. Such property is classical, and we just quickly outline the arguments that apply to our setting, since we were not able to find a precise reference to quote.

\begin{remark}[Compactness of $W^{1,1}_{\mathrm{loc}}$ in $L^1_{\mathrm{loc}}$]\label{rem:CompactnessBell}
Let $(X,\dist,\meas)$ be an $\CD(K,N)$ space with $K\leq 0$ and $N<+\infty$. Let us fix $R>0$ and $x\in X$, and denote $B:=B_R(x)$. We want to show that if $f_i\in \mathrm{Lip}_{\mathrm{loc}}(X)$ are such that 
$$
\sup_{i}\left(\|f_i\|_{L^1(B)}+\||Df_i|\|_{L^1(B)}\right)<+\infty,
$$
hence there exists $f\in L^1(B)$ such that $f_i\to f$ in $L^1(B)$. 

Indeed, let us first notice that, as a consequence of Bishop--Gromov comparison, the measure $\meas$ is locally doubling. Moreover, as a consequence of \cite[Theorem 1.1]{Rajala12} and Bishop--Gromov comparison, we have that, for every $i\in\mathbb N$, the pair $(f_i,|Df_i|)$ satisfies a $1$-Poincare inequality in $B$ according to the definition in \cite[Equation (5)]{HajlaszKoskela}. 

Moreover, an application of \cite[Theorem 9.7]{HajlaszKoskela} where $\Omega,s$ there become $B,n$ here, gives us that 
\begin{equation}\label{eqn:UNASOBOLEV}
\inf_{c\in\mathbb R}\left(\frac{1}{\meas(B)}\int_B|f_i-c|^{n/(n-1)}\de \meas\right)^{(n-1)/n}\leq \frac{C_1}{\meas(B)}\int_B |Df_i|\de \meas,
\end{equation}
for every $i\in\mathbb N$, where $C_1$ is some constant depending on $B$. Notice that \cite[Theorem 9.7]{HajlaszKoskela} can be applied since the hypotheses are met due to the fact that Bishop--Gromov comparison theorem holds, balls are John domains in $X$ (cf. \cite[Corollary 9.5]{HajlaszKoskela}) and $(f_i,|Df_i|)$ has the truncation property (cf. \cite[Theorem 10.3]{HajlaszKoskela}). As a consequence of \eqref{eqn:UNASOBOLEV} we have that, calling $\overline f_i:=(\int_B f_i\de \meas)/\meas(B)$, we get, 
$$
\|f_i-\overline{f_i}\|_{L^{n/(n-1)}(B)}\leq C_2\||Df_i|\|_{L^1(B)},
$$
for every $i>0$, and for some constant $C_2$ depending on $B$. From the previous one easily gets, by using the triangular inequality, that 
$$
\|f_i\|_{L^{n/(n-1)}(B)}\leq C_3\|f_i\|_{L^1(B)}+C_2\||Df_i|\|_{L^1(B)}.
$$
Now, since \cite[Equation (46)]{HajlaszKoskela} is met, we can apply \cite[Theorem 8.1]{HajlaszKoskela} and conclude that, up to subsequences, $f_i\to f$ in $L^1(B)$.
\end{remark}

\begin{lemma}\label{lem:CompactnessPerimeter}
Let $(X,\dist,\meas)$ be a $\CD(K,N)$ space with $K\leq 0$ and $N<+\infty$. Let $g_i\in L^1_{\mathrm{loc}}(X,\meas)$ such that, for every open bounded set $\Omega\subseteq X$, we have 
$$
\sup_{i\in\mathbb N}\left(\|g_i\|_{L^1(\Omega,\meas)}+|Dg_i|(\Omega)\right)<+\infty.
$$

Hence there exists $g\in L^1_{\mathrm{loc}}(X,\dist,\meas)$ such that $g_i\to g$ up to subsequences in $L^1_{\mathrm{loc}}(X,\meas)$.
\end{lemma}
\begin{proof}
Let us fix $x\in X$ and $B:=B_R(x)$, $2B:=B_{2R}(x)$ for some $R>0$. Choosing $\Omega=2B$, 
from the hypothesis and from the definition of the total variation $|Dg_i|$ we have the existence of $f_i\in \mathrm{Lip}_{\mathrm{loc}}(2B)$ such that 
$$
\sup_{i\in\mathbb N}\left(\|f_i\|_{L^1(B,\meas)}+\||Df_i|\|_{L^1(2B,\meas)}\right)<+\infty, \quad \text{and for all $i$ we have} \|f_i-g_i\|_{L^1(B,\meas)}\leq 1/i.
$$
Hence \cref{rem:CompactnessBell} implies that there is a subsequence, still denoted by $f_i$, and $g\in L^1(B,\meas)$ such that we have $f_i\to g$ in $L^1(B,\meas)$. Finally, the choice of $f_i$, we also have $g_i\to g$ in $L^1(B,\meas)$. 

By taking $R_i\to+\infty$ and using a diagonal argument, we get the existence of $g\in L^1_{\mathrm{loc}}(X,\meas)$ as in the statement.
\end{proof} 

We can now prove the main result of this section on the mass splitting of perimeter minimizing sequences.

\begin{theorem}\label{thm:RitoreRosalesNonSmooth}
    Let $K\leq 0$, and let $N\geq 2$. Let $(X,\dist,\mathcal{H}^N,x)$ be a pointed noncompact $\RCD(K,N)$ space. Let us fix a minimizing (for the perimeter) sequence $\{\Omega_k\}_{k\in\mathbb N}$ of Borel sets of volume $V$. Hence, up to passing to a subsequence, there exists a finite perimeter set $\Omega\subseteq X$ and sequence of finite perimeter sets $\{\Omega_k^c\}_{k\in\mathbb N}$ and $\{\Omega_k^d\}_{k\in\mathbb N}$ such that the following holds.
    \begin{itemize}
        \item[(i)] There exists a diverging sequence of radii $\{r_k\}_{k\geq 1}$ such that $$
        \Omega_k^c:=\Omega_k\cap B_{r_k}(x),\qquad \Omega_k^d:=\Omega_k\setminus B_{r_k}(x).
        $$
        \item[(ii)] We have that 
        $$
        \lim_{k\to+\infty}\left(P(\Omega_k^c)+P(\Omega_k^d)\right)=I_X(V).
        $$
        \item[(iii)] $\Omega$ is an isoperimetric region and 
        \begin{equation}\label{eqn:LaPrimaELaSeconda}
        \lim_{k\to +\infty}\mathcal{H}^N(\Omega_k^c)=\mathcal{H}^N(\Omega), \qquad \text{and}, \qquad \lim_{k\to +\infty}P(\Omega_k^c)=P(\Omega).
        \end{equation}
    \end{itemize}
\end{theorem}

\begin{proof}
Let $f_k:=\chi_{\Omega_k}$. From \cref{lem:CompactnessPerimeter} we get that there exists $g\in L^1_{\mathrm{loc}}(X,\mathcal{H}^N)$ such that, up to passing to subsequences, $f_k\to g$ in $L^1_{\mathrm{loc}}(X,\mathcal{H}^N)$. Hence $g=\chi_\Omega$ for some Borel set $\Omega\subseteq X$, and $\mathcal{H}^N(\Omega)\leq V$ by Fatou's Lemma. Moreover, by lower semicontinuity, we get that $P(\Omega)\leq I_X(V)$. Hence if $\mathcal{H}^N(\Omega)=V$ we get the conclusion with $\Omega_k^c:=\Omega_k$ and $\Omega_k^d:=\emptyset$.
Suppose from now on that $\mathcal{H}^N(\Omega)<V$.

Let us set $r_0:=0$. We want to define $r_{k}$ inductively in such a way that, up to eventually passing to a subsequence, we have 
\begin{itemize}
    \item[(a)] $\{r_k\}_{k\in\mathbb N}$ is divergent and $r_{k}-r_{k-1}\geq k$ for every $k\geq 1$; 
    \item[(b)] $P(\Omega_k\cap B_{r_k}(x))=P(\Omega_k,B_{r_k}(x))+P(B_{r_k}(x),\Omega_k^{(1)})$ for every $k\geq 1$;
    \item[(c)] $P(\Omega_k\setminus B_{r_k}(x))=P(\Omega_k,X\setminus B_{r_k}(x))+P(B_{r_k}(x),\Omega_k^{(1)})$ for every $k\geq 1$;
    \item[(d)] $P(B_{r_k}(x),\Omega_k^{(1)})\leq V/k$ for every $k\geq 1$;
    \item[(e)] there holds
    $$
    \int_{B_{r_k}(x)}|\chi_{\Omega_k}-\chi_{\Omega}|\de \mathcal{H}^N\leq 1/k,
    $$
    for every $k\geq 1$. 
\end{itemize}
First fix $r'_k\geq r_{k-1}+2k$. Hence, up to passing to a subsequence in $k$, we can assume (e) holds with $r'_k$ in place of $r_k$. This can be done since $\chi_{\Omega_k}\to\chi_{\Omega}$ in $L^1_{\rm{loc}}(X,\mathcal{H}^N)$. Moreover notice that, from the coarea formula we have that 
$$
\int_I P(B_s(x),\Omega_k^{(1)})\de s\leq \mathcal{H}^N(\Omega_k^{(1)})=V,
$$
on every interval $I\subseteq \mathbb R$.
Hence, by the previous estimate, we can find an $r_k\in[r'_k-k,r'_k]$ such that (d) holds. Moreover, such an $r_k$ can be taken in such a way that (b) and (c) hold, see \cite[Corollary 2.6]{AntonelliPasqualettoPozzetta}. Moreover, since $r_k\leq r'_k$, (e) still holds. Finally, since $r_k\geq r'_k-k\geq r_{k-1}+k$, we get that the sequence diverges and (a) holds as well. 

Let us prove item (ii). We have, for every $k\geq 1$, that 
$$
I_X(V)\leq P(\Omega_k)\leq P(\Omega_k^c)+P(\Omega_k^d)\leq P(\Omega_k)+2P(B_{r_k}(x),\Omega_k^{(1)}),
$$
where we are using item (b), (c), and (d) proved above. Taking $k\to +\infty$ we thus get item (ii). 

Let us now prove the first convergence in \eqref{eqn:LaPrimaELaSeconda}. By the triangle inequality we have 
$$
|\mathcal{H}^N(\Omega_k^c)-\mathcal{H}^N(\Omega)|\leq \int_{B_{r_k}(x)}|\chi_{\Omega_k}-\chi_{\Omega}|\de\mathcal{H}^N + \mathcal{H}^N(\Omega\setminus B_{r_k}(x)).
$$
Taking $k\to +\infty$, and taking into account item (a), and (e) above, we get the sought convergence.

Let us now prove that $P(\Omega)\leq \liminf_{k\to +\infty} P(\Omega_k^c)$. Indeed, this is a consequence of the fact that, by construction, for every $k\geq 1$ we have 
$$
\int_{B_{r_k}(x)}|\chi_{\Omega_k}-\chi_{\Omega}|\de\mathcal{H}^N=\int_{B_{r_k}(x)}|\chi_{\Omega_k^c}-\chi_{\Omega}|\de\mathcal{H}^N.
$$
Hence the previous equality, together with item (a) above, shows that $\chi_{\Omega_k^c}\to\chi_{\Omega}$ in $L^1_{\mathrm{loc}}(X,\mathcal{H}^N)$. Hence, by lower semicontinuity of the perimeter, we get the sought claim.

Let us call $\mathcal{H}^N(\Omega):=V_0$. Let us now prove that $\Omega$ is an isoperimetric region. If not, there exists a Borel set $E$ such that $\mathcal{H}^N(E)=V_0$ and $P(E)<P(\Omega)$. Again by the coarea formula, and since item (a) holds, there exists, for every $k\geq 1$, a $t_k\in [r_{k-1},r_k]$ such that 
$$
P(B_{t_k}(x),E^{(1)})\leq V_0/k.
$$
Moreover, $t_k$ can be chosen in such a way that $P(E\cap B_{t_k}(x))=P(E,B_{t_k}(x))+P(B_{t_k}(x),E^{(1)})$, see \cite[Corollary 2.6]{AntonelliPasqualettoPozzetta}.
Notice that if $E_k:=E\cap B_{t_k}(x)$ we have $\mathcal{H}^N(E_k)\to V_0$ as $k\to +\infty$.

Let us now fix $p\in E^{(1)}$ and $q\in E^{(0)}$. Recall by \cref{rem:PerimeterMMS2} that $P(B_r(p))\leq s(N,K/(N-1),r)$ and $P(B_r(q))\leq s(N,K/(N-1),r)$ hold for every $r>0$.

Since $\mathcal{H}^N(\Omega_k^d)\to V-V_0$ as $k\to +\infty$, we have that for $k$ large enough there exist $\rho_{p,k},\rho_{q,k}<1$ such that 
$$
\mathcal{H}^N((E_k\cup B_{\rho_{q,k}}(q))\setminus B_{\rho_{p,k}}(p))=V-\mathcal{H}^N(\Omega_k^d),
$$
and moreover $\rho_{q,k},\rho_{p,k}\to 0$ as $k\to +\infty$. This is done by arguing verbatim as in the proof of \cite[Lemma 3.1]{AFP21}.
Let us call $E'_k:=(E_k\cup B_{\rho_{q,k}}(q))\setminus B_{\rho_{p,k}}(p)$ and $T_k:=E'_k\cup \Omega_k^d$. Hence
\begin{equation}
\begin{split}
P(T_k)&\leq P(E_k)+P(B_{\rho_{q,k}}(q))+P(B_{\rho_{p,k}}(p))+P(\Omega_k^d) \leq P(E,B_{t_k}(x)) \\&+P(B_{t_k}(x),E^{(1)})+s(N,K/(N-1),\rho_{q,k})+s(N,K/(N-1),\rho_{p,k})+P(\Omega_k^d) \\
&\leq P(E)+V_0/k+s(N,K/(N-1),\rho_{q,k})+s(N,K/(N-1),\rho_{p,k})+P(\Omega_k^d).
\end{split}
\end{equation}
Hence, taking $k\to +\infty$ in the previous inequality, we get 
\begin{equation}
\begin{split}
I_X(V)\le \liminf_{k\to +\infty}P(T_k)&\leq P(E)+\liminf_{k\to+\infty}P(\Omega_k^d)<P(\Omega)+\liminf_{k\to+\infty}P(\Omega_k^d)\\&\leq \liminf_{k\to+\infty}P(\Omega_k^c)+\liminf_{k\to+\infty}P(\Omega_k^d)\leq I_X(V),
\end{split}
\end{equation}
where in the previous inequality we are using $P(\Omega)\leq \liminf_{k\to+\infty}P(\Omega_k^c)$, that we proved above, and the item (ii). This is a contradiction.

In order to conclude the proof we should just prove the second of \eqref{eqn:LaPrimaELaSeconda}. Moreover, since we already proved $P(\Omega)\leq\liminf_{k\to+\infty}P(\Omega_k^c)$, we just need to prove that $P(\Omega)\geq\limsup_{k\to+\infty}P(\Omega_k^c)$. If not, up to subsequence 
$$
P(\Omega) < \lim_{k\to+\infty}P(\Omega_k^c).
$$
In order to reach a contradiction from the previous inequality, and thus conclude the proof, it suffices the run verbatim the argument we used above to prove that $\Omega$ is an isoperimetric region, but with $\Omega$ in place of $E$.
\end{proof}

\begin{remark}\label{rem:PerimetriAZero}
As one can notice by a careful inspection of the previous reasoning, the proof of \cref{thm:RitoreRosalesNonSmooth} is likely to hold in a metric measure setting more general than the one of $\RCD(K,N)$ spaces $(X,\dist,\haus^N)$. A key property is the vanishing, as $r\to 0$, of the perimeter of the balls of radius $r$ around points of the metric measure space. Since this goes beyond the scope of this note, we will not discuss it further, but it will be subject of further investigations.
\end{remark}

\section{The isoperimetric problem via direct method}\label{sec:MAIN}

We shall need the following useful result stating the equality between the isoperimetric profile of a space $X$ and the one of an arbitrary disjoint union of $X$ with some of its pmGH limits at infinity.

\begin{proposition}\label{prop:ProfileDecomposition}
Let $K\leq 0$ and let $N\geq 2$. Let $(X,\dist,\mathcal{H}^N)$ be a noncompact $\RCD(K,N)$ space. Assume there exists $v_0>0$ such that $\mathcal{H}^N(B_1(x))\geq v_0$ for every $x\in X$. Let $\{p_{i,j} \st i \in \N\}$ be a sequence of points on $X$, for $j=1,\ldots,\overline{N}$ where $\overline{N}\in\N \cup \{+\infty\}$. Suppose that each sequence $\{p_{i,j}\st i \in \N\}$ is diverging along $X$ and that $(X,\dist,\mathcal{H}^N,p_{i,j})$ converges in the pmGH sense  to a pointed $\RCD(K,N)$ space $(X_j,\dist_j,\mathcal{H}^N,p_j)$. Defining
\begin{equation}\label{eqn:GeneralizedIsoperimetricProfile}
I_{X\sqcup_{j=1}^{\overline N}X_j}(v):=\inf\left\{P(E)+\sum_{j=1}^{\overline N}P(E_j):E\subseteq X,E_j\subseteq X_j, \mathcal{H}^N(E)+\sum_{j=1}^{\overline N}\mathcal{H}^N(E_j)=v\right\},
\end{equation}
it holds $I_{X\sqcup_{j=1}^{\overline N} X_j}(v) = I_X(v)$ for any $v>0$.
\end{proposition}

\begin{proof}
The inequality $I_{X\sqcup_{j=1}^{\overline N} X_j}\le I_X$ trivially follows from the definition of isoperimetric profile, so we need to prove the opposite inequality. By a minor variant of \cref{lem:ProfileOnBoundedSets}, cf. \cite[Lemma 2.19]{AFP21}, one can prove that
\begin{equation*}
\begin{split}
 I_{X\sqcup_{j=1}^{\overline N}X_j}(v):=\inf\bigg\{P(E)+\sum_{j=1}^{\overline N}P(E_j)
 \st &E\subset X,E_j\subset X_j, \mathcal{H}^N(E)+\sum_{j=1}^{\overline N}\mathcal{H}^N(E_j)=v, \\
 &E,E_j\text{ bounded for any }j\bigg\}.
\end{split}
\end{equation*}
Now let $v>0$. Let us prove the result in the case $\overline{N}=+\infty$, as it will be clear that the case of $\overline{N}\in \N$ follows by the analogous simplified argument.
For any $\eps>0$ there are bounded sets $E\subset X,E_j\subset X_j$ such that
\[
\mathcal{H}^N(E)+\sum_{j=1}^{\infty}\mathcal{H}^N(E_j)=v,
\qquad
P(E)+\sum_{j=1}^{\infty}P(E_j) \le  I_{X\sqcup_{j=1}^{\infty}X_j}(v) + \eps.
\]
For any $\delta >0$ there is $j_\delta$ such that $v(j_\delta)\eqdef \sum_{j_\delta+1}^\infty \haus^N(E_j)<\delta$ and $\sum_{j_\delta+1}^\infty P(E_j)<\delta$. Taking $\delta<v_0$, since $\{p_{i,j}\}_i$ is diverging, repeatedly applying \cref{lem:ApproxInMassa} we can find disjoint sets $B_{r_\delta}(o),F_1,\ldots,F_{j_\delta} \subset X$ contained in balls mutually located at positive distance such that
\[
\haus^N(B_{r_\delta}(o))= v(j_\delta),\qquad
\haus^N(F_j)=\haus^N(E_j) ,
\qquad
P(F_j)\le P(E_j)+\eps/j_\delta,
\]
for any $j=1,\ldots,j_\delta$, and $E$ is located at positive distance from $B_{r_\delta}(o)$. Hence $E\cup B_{r_\delta}(o) \cup F_1 \cup \ldots\cup F_{j_\delta}$ has measure equal to $v$ and
\[
\begin{split}
I_X(v) &\le P(E) + P(B_{r_\delta}(o)) + \sum_{j=1}^{j_\delta} P(F_j) \le  P(E) + P(B_{r_\delta}(o))+\sum_{j=1}^{j_\delta}P(E_j) +\eps\\& \le s(N,K/(N-1),r_\delta)+
I_{X\sqcup_{j=1}^{\infty}X_j}(v) + 2\eps-\delta .
\end{split}
\]
Since $\eps, \delta$ are arbitrary, and $r_\delta\to0$ as $\delta\to0$, we can let $\delta\to 0$ and then $\eps\to0$ to get the desired inequality.
\end{proof}

\subsection{Proof of the main results}

We are ready for proving our main results leading to the description of the behavior of perimeter minimizing sequences on $\RCD$ spaces. We start from the generalized compactness result in \cref{thm:GeneralizedCompactness}.

\begin{proof}[Proof of \cref{thm:GeneralizedCompactness}]
Up to subsequence, we may assume $P(E_i)\to A$, $\mathcal{H}^N(E_i)\to W$. From the isoperimetric inequality for small volumes \cref{prop:IsopVolumiPiccoli}, we get that if $A=0$ then $W=0$, and then the statement trivializes. Hence we may assume $A>0$. Moreover we may assume $W>0$ as well without loss of generality. Furthermore, we shall identify $E_i$ with the set of density $1$ points $E_i^{(1)}$ in this proof for ease of notation.

The proof follows the same strategy of \cite[Theorem 4.6]{AFP21}, except that here we deal with a sequence of different ambient spaces $X_i$ and the sequence $E_i$ has no a priori minimality properties. However, several arguments do not make use of such properties, and then an analogous proof can be carried out also in this setting. We refer the reader to \cite[Theorem 4.6]{AFP21} for detailed proofs of the first three steps below, which are based on an iterative application of \cref{lem:CoCo} together with \cref{lem:MasLowBound}.

\begin{itemize}
    \item[Step 1.] Up to passing to a subsequence in $i$, we claim that for any $i$ there exist an increasing sequence of natural numbers $N_i\ge1$ with limit $\overline{N}\eqdef \lim_i N_i \in \N\cup\{+\infty\}$, points $p_{i,1},\ldots, p_{i,N_i} \in X_i$ for any $i$, radii $R_j\ge1$ and numbers $\eta_j\in(0,1]$ defined for $j<\overline{N}$, and, if $\overline{N}<+\infty$, also a sequence of radii $R_{i,\overline{N}}\ge 1$, such that
    \begin{equation}\label{eq:Step1}
        \begin{split}
             \lim_i \dist_i(p_{i,j},p_{i,k}) =+\infty& \qquad\forall\,j\neq k<\overline N+1  ,\\
             \dist_i(p_{i,j},p_{i,k})\ge R_j+R_k+2& \qquad 
              \forall\,i\in \N, \forall\, j\neq  k\leq N_i, \\
              &\qquad j\neq \overline{N}, k \neq \overline{N} ,\\
            \exists \, \lim_i \mathcal{H}^N(E_i \cap B_{R_j}(p_{i,j})) = w_j' >0  & \qquad \forall\, j<\overline N,\\
             P(E_i,\partial B_{R_j}(p_{i,j})) = 0& \qquad \forall\,i, \forall\, j\le N_i , j \neq\overline N , \\
            P(B_{R_j}(p_{i,j}),E_i) \le \frac{\eta_j}{2^j} & \qquad \forall\,i, \forall\, j\le N_i , j \neq\overline N, \\
            \text{if also $\overline{N}=+\infty$, then }  W\ge  \lim_i  \sum_{j=1}^{N_i} \mathcal{H}^N( E_i\cap  B_{R_j}(p_{i,j}) ), & \\
            \text{if instead $\overline{N}<+\infty$, then $\lim_i R_{i,\overline{N}} = + \infty$,}&\\ 
             P(E_i,\partial B_{R_{i,\overline N}}(p_{i,\overline N})) = 0 & \qquad \forall\,i \st N_i=\overline{N}, \\
            P(B_{R_{i,\overline N}}(p_{i,\overline N}),E_i) \le \frac{1}{2^{\overline N}} & \qquad \forall\,i \st N_i=\overline{N},\\
            \dist_i\left(\partial B_{R_{i,\overline{N}}}(p_{i,\overline{N}}) , \partial B_{R_{j}}(p_{i,j}) \right) > 2& \qquad \forall\,i\st N_i=\overline{N}, \forall \,j\neq\overline{N},\\
            W = \lim_i \mathcal{H}^N\left(  B_{R_{i,\overline{N}}}(p_{i,\overline{N}}) \cap \left( E_i  \setminus \bigcup_{j=1}^{\overline{N}-1} B_{R_j}(p_{i,j}) \right)\right) & + \sum_{j=1}^{\overline{N}-1} \mathcal{H}^N(  E_i \cap  B_{R_j}(p_{i,j})) .\\
        \end{split}
    \end{equation}

\item[Step 2.] We claim that if $\overline{N}=+\infty$ then
    \begin{equation}\label{eq:Step2a}
        W= \lim_i \sum_{j=1}^{N_i} \mathcal{H}^N( E_i \cap  B_{R_j}(p_{i,j}) ).
    \end{equation}
    Moreover, we claim that, up to passing to a subsequence in $i$, there exist sequences of radii $\{T_{i,j}\}_{i \in \N}$ such that $T_{i,j} \in (R_j,R_j+1)$ for any $j<\overline{N}$, and $T_{i,\overline{N}} \in (R_{i,\overline{N}}, R_{i,\overline{N}}+1)$ if $\overline{N}<+\infty$, such that the following hold \begin{equation}\label{eq:AsympMassDecomp1}
    \begin{split}
        B_{T_{i,j}}(p_{i,j}) \cap B_{T_{i,k}}(p_{i,k}) = \emptyset & \qquad \forall\,i\in \N, \forall\, j\neq  k\leq N_i,  j\neq \overline{N}, k \neq \overline{N} ,\\
        \lim_i T_{i,j} = T_j < + \infty & \qquad \forall\,j < \overline{N}, \\
        P(E_i, \partial B_{T_{i,j}}(p_{i,j}) ) = 0
         & \qquad \forall i\in\mathbb N,\,\, \forall 1\leq j \leq N_i , \, j\neq \overline{N}\\
        \text{if also $\overline{N}<+\infty$, then $\lim_i T_{i,\overline{N}} = + \infty$ and }&\\ 
        \partial B_{T_{i,\overline{N}}}(p_{i,\overline{N}}) \cap \partial B_{T_{i,j}}(p_{i,j}) =\emptyset& \qquad \forall\,i\st N_i=\overline{N}, \forall \,j<\overline{N}, \\
        P(E_i, \partial B_{T_{i,\overline{N}}}(p_{i,\overline{N}}) ) =0 & \qquad \forall \, i \in \N.
    \end{split}
    \end{equation}
    Moreover
    \begin{equation}\label{eq:Step2b}
        \lim_i \sum_{j=1}^{N_i} P( B_{T_{i,j}}(p_{i,j}),E_i) = 0.
    \end{equation}
    
\item[Step 3.] We claim that letting $E_i^v \eqdef E_i \setminus \bigcup_{j=1}^{N_i} \left(E_i \cap B_{T_{i,j}}(p_{i,j})\right)$, then
    \begin{equation}\label{eq:Step3a}
            \lim_i \mathcal{H}^N(E_i^v)=0,
    \end{equation}
    and that, if $\overline{N}=+\infty$, then
    \begin{equation}\label{eq:Step3b}
            W= \lim_i \sum_{j=1}^{N_i} \mathcal{H}^N( E_i \cap  B_{T_{i,j}}(p_{i,j}) ) = \sum_{j=1}^{+\infty} \lim_i \mathcal{H}^N( E_i \cap  B_{T_{i,j}}(p_{i,j}) ).
    \end{equation}

\item[Step 4.] Denoting $E_{i,j}:= E_i \cap B_{T_{i,j}}(p_{i,j})$ and $ E_{i,\overline N}\eqdef B_{T_{i,\overline{N}}}(p_{i,\overline{N}}) \cap E_i  \setminus \bigcup_{j=1}^{\overline{N}-1} B_{T_{i,j}}(p_{i,j}) $ if $\overline{N}<+\infty$, we claim that for any $j < \overline{N}+1$ there exists an $\RCD(K,N)$ space $(Y_j,\dist_{Y_j},\mathcal{H}^N)$, points $p_j \in Y_j$ and Borel sets $F_j\subset Y_j$ such that
    \begin{equation}\label{eq:Itemiii1}
        \begin{split}
            (X_i,\dist_i,\mathcal{H}^N,p_{i,j}) \xrightarrow[i]{} (Y_j,\dist_{Y_j},\mathcal{H}^N, p_j) & \qquad \text{in the $pmGH$ sense for any $j$},\\
            E_{i,j} \xrightarrow[i]{} F_j \subset Y_j& \qquad \text{in the $L^1$-strong sense for any $j<\overline{N}$}, \\
            \mathcal{H}^N(E_{i,j}) \xrightarrow[i]{} \mathcal{H}^N(F_j)& \qquad \forall \,j<\overline{N} \\
            \liminf_i P(E_{i,j}) \ge P (F_j)& \qquad \forall \,j<\overline{N},
        \end{split}
    \end{equation}
    and if $\overline{N}<+\infty$ then
    \begin{equation}\label{eq:Itemiii2}
        \begin{split}
            E_{i,\overline{N}}
            \xrightarrow[i]{} F_{\overline{N}} \subset Y_{\overline{N}}  & \qquad \text{in the $L^1$-strong sense},\\
            \mathcal{H}^N(E_{i,\overline N})
            \xrightarrow[i]{} \mathcal{H}^N(F_{\overline{N}}), & \\
            \liminf_i P(E_{i,\overline N})
            \geq P (F_{\overline{N}}).
        \end{split}
    \end{equation}
Indeed, using Gromov Precompactness Theorem and \cref{prop:SemicontinuitaAmbrosioBrueSemola}, together with the suitable diagonal argument, one easily gets spaces $(Y_j,\dist_{Y_j},\haus^N,p_j)$ and sets $F_j\subset Y_j$ as in \eqref{eq:Itemiii1} and \eqref{eq:Itemiii2}. The $L^1$-strong convergence of $E_{i,j}$ to $F_j$ with $j<\overline N$ is guaranteed by \cref{prop:SemicontinuitaAmbrosioBrueSemola} because $T_{i,j}\to T_j$ as $i\to +\infty$. Let us show that also $E_{i,\overline{N}}$ converges in $L^1$-strong to $F_{\overline N}$. This easily follows from the last line of \eqref{eq:Step1} and the definitions of $T_{i,j}$. Putting together the last line of \eqref{eq:Step1}, \eqref{eq:Step3b}, \eqref{eq:Itemiii1}, \eqref{eq:Itemiii2}, we deduce \eqref{eq:GeneralizedCMPcontM}.

Let us now show \eqref{eq:GeneralizedCMPsemicontP}. Assume $\overline{N}=+\infty$. For any $M\in\N$ we have
\[
\begin{split}
    \sum_{j=1}^M P(F_j) &\le \liminf_i \sum_{j=1}^M P(E_{i,j}) = \liminf_i \sum_{j=1}^M P(E_i, B_{T_{i,j}}(p_{i,j})) + P(B_{T_{i,j}}(p_{i,j}), E_i) \\&\le
    \liminf_i P(E_i) + \sum_{j=1}^{N_i}P(B_{T_{i,j}}(p_{i,j}), E_i),
\end{split}
\]
and using \eqref{eq:Step2b} and letting $M\to+\infty$, we obtain \eqref{eq:GeneralizedCMPsemicontP}. In case $\overline{N}<+\infty$, the analogous computation yields \eqref{eq:GeneralizedCMPsemicontP}.

\item[Step 5] Suppose now that $E_i$ is isoperimetric for any $i$, we want to prove the last part of the statement. Assume by contradiction that $F_{j_0}$ is not isoperimetric for some $j_0$, then there exists a bounded set $\widetilde{F}_{j_0}\subset Y_{j_0}$ such that $P(\widetilde{F}_{j_0}) \le P(F_{j_0})-\eps$ and $\haus^N(\widetilde{F}_{j_0}) =\haus^N(F_{j_0})$, for $\eps>0$. By \cref{prop:SemicontinuitaAmbrosioBrueSemola} there exists a sequence $\widetilde{E}_{i,j_0}\subset B_r(p_{i,j_0})\subset X_i$, for a fixed $r>0$, such that $\haus^N(\widetilde{E}_{i,j_0})\to \haus^N(\widetilde{F}_{j_0})$ and $P(\widetilde{E}_{i,j_0})\to P(\widetilde{F}_{j_0})$.
Therefore we can consider the sequence
\begin{equation}
    \widetilde{E}_i \eqdef \left( E_i \setminus \left( E_i^v \cup E_{i,j_0} \right) \right) \cup \left(\widetilde{E}_{i,j_0} \cup B_{\rho_i}(q_i) \setminus B_{\tilde \rho_i}(\tilde q_i)\right),
\end{equation}
where $\rho_i,\tilde\rho_i\ge 0$, and $B_{\rho_i}(q_i),  B_{\tilde \rho_i}(\tilde q_i)\subset X_i$ are some balls such that $\haus^N(\widetilde{E}_i)=\haus^N(E_i)$, $\dist (B_{\rho_i}(q_i),  B_{\tilde \rho_i}(\tilde q_i))>0$, and $\sup_i \dist(q_i, p_{i,j_0}) + \dist(\tilde q_i, p_{i,j_0}) <+\infty$. By construction, taking into account \eqref{eq:Step1}, we have that $\rho_i, \tilde \rho_i\to 0$ and then volume and perimeter of the balls $B_{\rho_i}(q_i),  B_{\tilde \rho_i}(\tilde q_i)$ tend to zero by Bishop--Gromov comparisons. By \eqref{eq:Step2b} we obtain
\[
\begin{split}
    I_{X_i}(\haus^N(E_i)) &\le P(\widetilde{E}_i ) \\&=
    P(E_i) -P(E_i^v) - P(E_{i,j_0}) +
    P(\widetilde{E}_{i,j_0}) + P(B_{\rho_i}(q_i)) + P(B_{\tilde \rho_i}(\tilde q_i))+ o(1) \\
    &= I_{X_i}(\haus^N(E_i)) -P(E_i^v) - P(E_{i,j_0}) +
    P(\widetilde{E}_{i,j_0}) + o(1).
\end{split}
\]
Since $\limsup_i  - P(E_{i,j_0}) + P(\widetilde{E}_{i,j_0}) \le -P(F_{j_0}) + P(\widetilde{F}_{j_0}) \le -\eps$, we get a contradiction. Moreover, by the very same computation, we also deduce that
\begin{equation}\label{eq:PerimetroVanishZero}
    \lim_i P(E_i^v) =0,
\end{equation}
for otherwise we would get another contradiction. And taking into account \eqref{eq:PerimetroVanishZero} and exploiting the minimality of the $E_i$'s as before, one easily gets the continuity of the sequence of perimeters in \eqref{eq:GeneralizedCMPcontP}.
\end{itemize}
\end{proof}

We can now prove the asymptotic mass decomposition result contained in \cref{thm:MassDecompositionINTRO}, which improves and generalizes previous results contained in \cite{Nar14} and \cite{AFP21}.

\begin{proof}[Proof of \cref{thm:MassDecompositionINTRO}]
The proof is an adaptation of \cite[Theorem 4.6]{AFP21}, which is based on previous ideas in \cite{Nar14}.
We divide the proof in steps. First of all, starting from the sequence $\{\Omega_i\}_{i\in\mathbb N}$, we know there exist sequences $\{\Omega_i^c\}_{i\in\mathbb N}$ and $\{\Omega_i^d\}_{i\in\mathbb N}$ as in \cref{thm:RitoreRosalesNonSmooth}. Let 
$$
W:=\lim_{i\to+\infty} \mathcal{H}^N(\Omega_i^d).
$$
If $W=0$ the theorem is trivially true since there is no loss of mass at infinity. 

By arguing verbatim as in the proof of \cref{thm:GeneralizedCompactness}, we get all the claims in Step 1, Step 2, Step 3, and Step 4 with the constant sequence $\{(X,\dist,\mathcal{H}^N)\}$ here in place of the sequence $\{(X_i,\dist_i,\mathcal{H}^N)\}_{i\in\mathbb N}$ there; with $\{\Omega_i^d\}_{i\in\mathbb N}$ here in place of $\{E_i\}_{i\in\mathbb N}$ there; with $\{Z_j\}_{1\leq j\leq \overline N}$ here in place of $\{F_j\}_{1\leq j\leq \overline N}$ there; with $\{(X_j,\dist_j,\mathcal{H}^N,p_j)\}_{1\leq j\leq \overline N}$ here in place of $\{(Y_j,\dist_{Y_j},\mathcal{H}^N,p_j)\}_{1\leq j\leq \overline N}$ there. 

For the sake of clarity let us stress that $\Omega_{i,j}^d=\Omega_i^d\cap B_{T_{i,j}}(p_{i,j})$, for every $i\in\mathbb N$ and every $1\leq j\leq N_i$ with $j\neq \overline N$; also $\Omega_{i,\overline N}^d=B_{T_{i,\overline N}}(p_{i,\overline N})\cap \Omega_i^d\setminus \cup_{j=1}^{\overline N-1}B_{T_{i,j}}(p_{i,j})$ if $\overline N<+\infty$; and $\Omega_i^v:=\Omega_i^d\setminus\cup_{j=1}^{N_i}\Omega_{i,j}^d$.

\begin{itemize}
    \item[Step 1.] We claim that
    \begin{equation}\label{eq:Step4a}
        \lim_i P(\Omega_i^v)=0,
    \end{equation}
    and that
    \begin{equation}\label{eq:Step4b}
        \lim_i P(\Omega_i^d) = 
        \lim_i \sum_{j=1}^{N_i}  P(\Omega_{i,j}^d).
    \end{equation}
    Moreover we claim that for every $1\leq j< \overline N +1$ we have 
    \begin{equation}\label{eq:CLAIMISOP}
        \lim_i P(\Omega_{i,j}^d) = 
        P(Z_j),
    \end{equation}
    and $Z_j$ is an isoperimetric region in $X_j$.
    
    In order to prove \eqref{eq:Step4a} and \eqref{eq:Step4b} we argue verbatim as in the proof of \cite[Step 4 of the Proof of Theorem 4.6]{AFP21}. Indeed, it is only needed the Bishop--Gromov comparison for volumes and perimeters, and mass splitting result in \cref{thm:RitoreRosalesNonSmooth}.
    
    In order to prove \eqref{eq:CLAIMISOP} and the fact that $Z_j$ are isoperimetric, we again argue verbatim as in \cite[Step 4 of the proof of Theorem 4.6]{AFP21}. Indeed, it is only needed the continuity of the isoperimetric profile, proved in \cref{lem:ProfileHolder}, the mass splitting result in \cref{thm:RitoreRosalesNonSmooth}, and the fact that isoperimetric regions are bounded, see \cref{thm:RegularityIsoperimetricSets}.
    
    At this stage of the proof we have shown that the first three items of the statement hold, indeed see the first line in \eqref{eq:Step1}, \cref{thm:RitoreRosalesNonSmooth} and \cref{thm:RegularityIsoperimetricSets}, \eqref{eq:Itemiii1} and \eqref{eq:Itemiii2}, and \eqref{eq:CLAIMISOP} and the line after \eqref{eq:CLAIMISOP}.
    
    \item[Step 2.]
    We now claim that fourth item of the Theorem holds. 
    
    Indeed we already know from \eqref{eq:Step3b} (and the last condition in \eqref{eq:Step1} when $\overline N<+\infty$) and from \cref{thm:RitoreRosalesNonSmooth} that
    \[
    W= \sum_{j=1}^{\overline{N}} \mathcal{H}^N (Z_j),
    \qquad
    V= \haus^N(\Omega) + W.
    \]
    Moreover by \cref{thm:RitoreRosalesNonSmooth}, \eqref{eq:Step4b}, and the third item of the statement we also deduce
    \[
    I_X(V) = \lim_i \left(P(\Omega_i^c) + P(\Omega_i^d)\right) \ge P(\Omega) + \sum_{j=1}^{\overline{N}} P(Z_j) = I_X(\haus^N(\Omega)) + \sum_{j=1}^{\overline{N}} I_{X_j}(\mathcal{H}^N(Z_j)).
    \]
    On the other hand, we are exactly in the hypotheses for applying \eqref{eq:InfinityInequalityIsopProfile}, that yields
    \[
    I_X(V) \le I_X(\haus^N(\Omega)) + \sum_{j=1}^{\overline{N}} I_{X_j}(\mathcal{H}^N(Z_j)).
    \]
    Hence equality holds, and this completes the proof of the fourth item of the Theorem.
    
\item[Step 3.]
We now prove that $\overline N$ is finite. Let us suppose for the sake of contradiction it is not.

For every $\ell\geq 1$ let us call $V_\ell:=V-\mathcal{H}^N(\Omega)-\sum_{j=1}^\ell \mathcal{H}^N(Z_j)$. Hence $V_\ell\to 0$ as a consequence of the second equality in \eqref{eq:UguaglianzeIntro}. 

Taking into account \cref{thm:RegularityIsoperimetricSets} we have that $Z_1$ has interior and exterior points in $X_1$. Let us now apply Item (i) of \cref{thm:VariazioniMaggi} to $Z_1\subset X_1$. We get that there exist constants $\eta_0:=\eta_0(Z_1)$ and $C:=C(Z_1)$ such that for every $\eta\in [0,\eta_0)$ there is a set $\widetilde Z_1\supset Z_1$ such that 
\[
\mathcal{H}^N(\widetilde Z_1)=\mathcal{H}^N(Z_1)+\eta, \qquad P(\widetilde Z_1)\leq C\eta + P(Z_1).
\]
Since $V_\ell\to 0$ as $\ell\to +\infty$ we have that there exists $\ell_0\in\mathbb N$ such that $V_\ell< \eta_0$ for every $\ell\geq \ell_0$. Let us now fix an arbitrary $\ell\geq \ell_0$.

Let us apply \cref{thm:VariazioniMaggi} as described above with $\eta=V_\ell$. We obtain $\widetilde Z_1\subset X_1$ such that
\begin{equation}\label{eqn:InequalityGoodVariation}
\mathcal{H}^N(\widetilde Z_1)=\mathcal{H}^N(Z_1)+V_\ell, \qquad P(\widetilde Z_1)\leq C V_\ell + P(Z_1).
\end{equation}
By \cref{prop:ProfileDecomposition} it follows that
\begin{equation}\label{eqn:SoughtInequalityOnI}
    I_X(V)\leq P(\Omega)+P(\widetilde Z_1) + \sum_{j=2}^\ell P(Z_j).
\end{equation}
We now aim at proving that 
\begin{equation}\label{eqn:AnotherInequality}
    I_X(V_\ell)+P(\Omega)+\sum_{k=1}^\ell P(Z_k)\leq I_X(V).
\end{equation}
Indeed, from the first equality in \eqref{eq:UguaglianzeIntro} we get that 
$$
I_X(V)=P(\Omega)+\sum_{k=1}^{+\infty}P(Z_k),
$$
and hence \eqref{eqn:AnotherInequality} is equivalent to proving 
$$
I_X(V_\ell) \leq \sum_{k=\ell+1}^{+\infty}P(Z_k),
$$
which directly comes from \eqref{eq:InfinityInequalityIsopProfile}, since $V_\ell=\sum_{k=\ell+1}^{+\infty}\mathcal{H}^N(Z_k)$ by the very definition of $V_\ell$ and the second equality in \eqref{eq:UguaglianzeIntro}, and since the $Z_k$'s are isoperimetric regions for every $k\geq \ell+1$.

Putting together the inequalities in \eqref{eqn:InequalityGoodVariation}, \eqref{eqn:SoughtInequalityOnI}, and \eqref{eqn:AnotherInequality}, we get that, for every $\ell\geq \ell_0$, the following inequality holds 
\begin{equation}\label{eqn:ContradictionVolume}
    I_X(V_\ell)\leq CV_\ell.
\end{equation}
We now aim at showing that \eqref{eqn:ContradictionVolume} gives a contradiction for $\ell\geq \ell_0$ large enough. Indeed, as a consequence of \cref{prop:IsopVolumiPiccoli}, we have that there exist constants $\vartheta,\Theta>0$ such that 
\begin{equation}\label{eqn:PrevIn}
\mathcal{H}^N(E)\leq \vartheta \qquad\Rightarrow\qquad \left(\mathcal{H}^N(E)\right)^{\frac{N-1}{N}}\leq \Theta P(E),
\end{equation}
for every set of finite perimeter $E$. Since $V_\ell\to 0$, for $\ell\geq \ell_0$ large enough we have $V_\ell\leq \vartheta$. Thus \eqref{eqn:PrevIn} implies that, for $\ell\geq \ell_0$ large enough, the following inequality holds
\begin{equation}\label{eqn:ContradictionVolume2}
I_X(V_\ell)\geq \Theta^{-1}V_\ell^{\frac{N-1}{N}}.
\end{equation}
Finally, for $\ell\geq \ell_0$ large enough, \eqref{eqn:ContradictionVolume} and \eqref{eqn:ContradictionVolume2} are in contradiction, since $V_\ell\to 0$, thus showing that $\overline N$ must be finite.
\end{itemize}
\end{proof}

We conclude this part by proving \cref{cor:LinearUpperBound}, which generalizes the previous estimate in \cite{Nar14}.

\begin{proof}[Proof of \cref{cor:LinearUpperBound}]
Exploiting \cref{thm:MassDecompositionINTRO} and \cref{thm:GeneralizedCompactness}, it is proved in \cite[Proposition 4.15]{AntonelliPasqualettoPozzettaSemola} that the isoperimetric profile of $\RCD(K,N)$ spaces is strictly subadditive for small volumes. More precisely, there exists $\eps=\eps(K,N,v_0)>0$ such that whenever $(Y,\dist_Y,\haus^N)$ is an $\RCD(K,N)$ with $N\ge2$ and $\inf_{y\in Y} \haus^N(B_1(y))\ge v_0>0$ then the profile $I_Y$ is strictly subadditive on $(0,2\eps)$.

Let $(X,\dist,\mathcal{H}^N)$ be a noncompact $\RCD(K,N)$ space such that $\mathcal{H}^N(B_1(x))\geq v_0$ for every $x\in X$. Let $V>0$ and let $\Omega_i\subset X$ be a minimizing (for the perimeter) sequence of bounded sets of volume $V$. Let $\overline{N}, \Omega, Z_j, X_j$ be given by applying \cref{thm:MassDecompositionINTRO}.
Let us also denote $Z_0\eqdef \Omega$ and $X_0\eqdef X$ here.

We claim that $\haus^N(Z_j)<\eps$ for at most one index $0\le j<\overline{N}+1$. Indeed, suppose by contradiction that $\haus^N(Z_k),\haus^N(Z_\ell)< \eps$ for some $k\neq \ell$ with $0\le j,k<\overline{N}+1$. Then we get
\[
\begin{split}
    I_X(V) &= \sum_{j=0}^{\overline{N}} P(Z_j) =
    I_{X_k}(\haus^N(Z_k))+I_{X_\ell}(Z_\ell) +\sum_{\stackrel{j=0}{j\neq k,\ell}}^{\overline{N}} I_{X_j}(Z_j) \\
    &\ge I_X(\haus^N(Z_k)) + I_X(\haus^N(Z_\ell)) +\sum_{\stackrel{j=0}{j\neq k,\ell}}^{\overline{N}} I_{X_j}(Z_j) \\
    &> I_X(\haus^N(Z_k) + \haus^N(Z_\ell))  +\sum_{\stackrel{j=0}{j\neq k,\ell}}^{\overline{N}} I_{X_j}(Z_j)\\
    &\ge I_X(V),
\end{split}
\]
where we applied \cref{prop:ComparisonIsoperimetricProfile} and the strict subadditivity of $I_X$ for small volumes, leading to a contradiction.

We therefore conclude that $\overline{N} \le 1 + V/\eps$.
\end{proof}

\section{An equivalent condition for compactness of minimizing sequences}\label{sec:Equivalent}
In this final chapter we present an equivalent condition that ensures compactness of the minimizing sequences for the isoperimetric problem. This equivalent condition involves a kind of isoperimetric profile at infinity that we are going to define.

\begin{definition}\label{def:Isoprofileatinfinity}
Let $(X,\dist,\meas)$ be a noncompact metric measure space. For any $V \in (0,\meas(X))$ we define
\[
I^\infty_X(V)\eqdef \inf \left\{\liminf_i P(E_i) \st \meas(E_i)\to V, \, \quad\text{$\forall$ ball B},\,\,E_i \cap B = \emptyset \quad\text{for $i$ large enough} \right\},
\]
\[
\widetilde I^\infty_X(V) \eqdef \inf \left\{ 
\liminf_i P(E_i) \st \meas(E_i)\to V, \, \lim_i \meas(E_i \cap B) = 0 \quad\text{$\forall$ ball B}
\right\}.
\] 
We call $I_X^\infty$ \emph{the isoperimetric profile at infinity of $(X,\dist,\meas)$} and $\tilde{I}_X^\infty$ \emph{the isoperimetric profile at infinity volumewise of $(X,\dist,\meas)$}.
\end{definition}

In the previous definition, we obviously have that $\widetilde I^\infty_X \le  I^\infty_X$. Under the following regularity assumptions on the space $X$, we can prove that the converse holds. The notion of isotropicity was introduced in
\cite[Definition 6.1]{AMP04}. We remark that all
\({\sf RCD}(K,N)\) spaces with \(N<\infty\), are isotropic PI spaces;
cf.\ \cite[Example 1.31(iii)]{BPR20}.

\begin{lemma}\label{lem:EqualityIsopInfiniti}
Let $(X,\dist,\meas)$ be a noncompact metric measure space. Assume that $(X,\dist,\meas)$ is PI and isotropic, then $\widetilde I^\infty_X =  I^\infty_X$.
\end{lemma}

\begin{proof}
Let $V\in(0,\meas(X))$. We need to prove that $\widetilde I^\infty_X (V) \ge  I^\infty_X(V)$. Let $E_i\subset X$ be such that $\meas(E_i)\to V$, $\lim_i \meas(E_i \cap B) = 0$ for any ball $B$, and $\lim_i P(E_i) = \widetilde I^\infty_X(V)$. Fix $o \in X$. For any $k \in \N$, since $\int_0^k P(B_r(o), E_i^{(1)}) \de r = \meas(E_i \cap B_k(o)) \to_i 0$, using also \cite[Corollary 2.6]{AntonelliPasqualettoPozzetta}, we find $i_k \in \N$ and $r_k \in (\tfrac{k}{2}, k)$ such that
\[
\meas(E_{i_k} \cap B_{r_k}(o)) <\frac1k,
\qquad
P( B_{r_k}(o), E^{(1)}_{i_k} ) <\frac1k,
\]
\[
P(E_{i_k}\setminus B_{r_k}(o)) = P(E_{i_k}, X\setminus B_{r_k}(o)) + P( B_{r_k}(o), E^{(1)}_{i_k} ) .
\]
Setting $F_k\eqdef E_{i_k} \setminus B_{r_k}(o)$ we get that
\[
I^\infty_X(V) \le \liminf_k P(F_k) \le \liminf_k P(E_{i_k}) = \widetilde I^\infty_X(V).
\]
\end{proof}

Let us now show a couple of inequalities involving $I_X$ and $I_X^\infty$.
\begin{lemma}\label{lem:DisugProfiliInfinito}
Let $(X,\dist,\meas)$ be a noncompact isotropic PI space. Assume that the isoperimetric profile $I_X$ is lower semicontinuous, then $I_X\le \widetilde I^\infty_X = I^\infty_X$ and
\begin{equation*}
    I_X(V) \le I_X(V_1) + I^\infty_X(V_2),
    \qquad
    \forall\, V_1+V_2 = V,
\end{equation*}
for any $V\in(0,\meas(X))$.
\end{lemma}

\begin{proof}
Assuming that $I_X$ is lower semicontinuous, then
\begin{equation*}
    I_X(V) = \inf\left\{ \liminf_i P(E_i) \st \lim_i \meas(E_i) = V \right\},
\end{equation*}
for any $V\in(0,\meas(X))$. Hence, recalling  \cref{lem:EqualityIsopInfiniti}, then $I_X\le \widetilde I^\infty_X = I^\infty_X$.

Now let $V_1,V_2,V$ be fixed as in the statement. Fix also $\eps>0$. Let $E\subset X$ such that $\meas(E)=V_1$ and $P(E)\le I_X(V_1) + \eps$. Let also $E_i\subset X$ such that $\meas(E_i)\to V_2$, $E_i \cap B=\emptyset$ for any ball $B$ for large $i$, and $P(E_i) \le I^\infty_X(V_2) + \eps$. Up to subsequence, there is a diverging sequence $r_i>0$ such that $E_i \cap B_{2r_i}(o)=\emptyset$ for any $i$, for some $o \in X$, and
\[
P(E \cap B_{r_i}(o) ) = P(E, B_{r_i}(o)) + P(B_{r_i}(o), E^{(1)}),
\qquad
P(B_{r_i}(o), E^{(1)})\le \frac1i,
\]
where we used \cite[Corollary 2.6]{AntonelliPasqualettoPozzetta} and the coarea formula. Letting $F_i\eqdef (E \cap B_{r_i}(o)) \cup E_i$, then $\meas(F_i)\to V$ and then
\[
I_X(V) \le \liminf_i P(F_i) \le \liminf_i \left(P(E) + P(E_i)\right) \le I_X(V_1) + I^\infty_X(V_2) + 2 \eps.
\]
By arbitrariness of $\eps$, the claim follows.
\end{proof}

The following theorem is a slight empowered version of \cref{thm:RitoreRosalesNonSmooth}, where we allow that the volumes of the minimizing sequence are not fixed, but just tend to a fixed value $V$. The proof is obtained arguing verbatim as in the proof of \cref{thm:RitoreRosalesNonSmooth}, but we need an additional hypothesis on the continuity of the isoperimetric profile. We skip the details for the sake of brevity.

\begin{theorem}\label{thm:RitoreRosalesModificato}
    Let $K\leq 0$, and let $N\geq 2$. Let $(X,\dist,\mathcal{H}^N,x)$ be a pointed noncompact $\RCD(K,N)$ space.
    
    Assume that the isoperimetric profile $I_X$ is continuous and let $\{\Omega_k\}_{k\in\mathbb N}$ be a sequence of Borel sets such that
    \[
    \lim_k \haus^N(\Omega_k) = V,
    \qquad
    \lim_k P(\Omega_k) = I_X(V),
    \]
    for some $V \in (0,\haus^N(X))$.
    
    Hence, up to passing to a subsequence, there exists a finite perimeter set $\Omega\subseteq X$ and a sequence of finite perimeter sets $\{\Omega_k^c\}_{k\in\mathbb N}$ and $\{\Omega_k^d\}_{k\in\mathbb N}$ such that the following holds.
    \begin{itemize}
        \item[(i)] There exists a diverging sequence of radii $\{r_k\}_{k\geq 1}$ such that $$
        \Omega_k^c:=\Omega_k\cap B_{r_k}(x),\qquad \Omega_k^d:=\Omega_k\setminus B_{r_k}(x).
        $$
        \item[(ii)] We have that 
        $$
        \lim_{k\to+\infty}\left(P(\Omega_k^c)+P(\Omega_k^d)\right)=I_X(V).
        $$
        \item[(iii)] $\Omega$ is an isoperimetric region and 
        \begin{equation}
        \lim_{k\to +\infty}\mathcal{H}^N(\Omega_k^c)=\mathcal{H}^N(\Omega), 
        \qquad
        \lim_{k\to +\infty}P(\Omega_k^c)=P(\Omega).
        \end{equation}
    \end{itemize}
\end{theorem}

We now conclude the section by proving the equivalent criterion for all minimizing sequences to precompactly converge to a solution of the isoperimetric problem without loss of mass.

\begin{proof}[Proof of {\cref{thm:AbstractCriterionCompactness}}]
We start from the implication $(1)\Rightarrow(2)$. Let  $\Omega_k \subset X$ such that $\haus^N(\Omega_k)\to V$ and $P(\Omega_k)\to I_X(V)$. Let $r_k, \Omega_k^c,\Omega_k^d,\Omega$ be given by applying \cref{thm:RitoreRosalesModificato}. If by contradiction $V_2\eqdef \lim_k \haus^N(\Omega_k^d)>0$, then
\[
I_X(V) = P(\Omega) + \lim_k P(\Omega^d_k) \ge I_X(\haus^N(\Omega)) + I^\infty_X(V_2),
\]
contradicting $(1)$.

It remains to prove that $(2)\Rightarrow (1)$. Let $V_1+V_2=V$ with $V_2\in (0,V]$. We consider the case $V_2<V$, the case $V_2=V$ being analogous and simpler. Suppose by contradiction that $I_X(V) \ge I_X(V_1) + I^\infty_X(V_2)$.

Let $E_j\subset X$ be such that $\haus^N(E_j)=V_1$ and $P(E_j)\le I_X(V_1)+\tfrac1j$, and $\tilde E_i\subset X$ such that $\haus^N(\tilde E_i)\to V_2$, $\tilde E_i \cap B=\emptyset$ for any ball $B$ for large $i$, and $P(\tilde E_i)\to I^\infty_X(V_2)$. Arguing similarly as in the proof of \cref{lem:DisugProfiliInfinito}, for any $j$ we can find $i_j\ge j$ and a diverging sequence $r_j>0$ such that
\[
P(E_j \cap B_{r_j}(o) ) = P(E_j, B_{r_j}(o)) + P(B_{r_j}(o), E_j^{(1)}),
\quad
\haus^N(E_j\setminus B_{r_j}(o) ) \le \frac1j,
\quad
P(B_{r_j}(o), E_j^{(1)})\le \frac1j,
\]
and $\tilde E_{i_j} \cap B_{2r_j}(o)=\emptyset$ for any $j$, for some $o \in X$.
Letting $F_j\eqdef (E_j \cap B_{r_j}(o)) \cup\tilde E_{i_j}$, then $\mathcal{H}^N(F_j)\to V$ and then
\[
I_X(V) \le \liminf_j P(F_j) \le I_X(V_1) + I^\infty_X(V_2) \le I_X(V),
\]
where the last inequality is the absurd hypothesis. Hence we can apply $(2)$ on the sequence $\{F_j\}$, which then admits a subsequence converging in $L^1(X)$. However, the sets $\tilde E_{i_j}$'s are diverging and their measure converges to $V_2>0$, hence no subsequences of $\{F_j\}$ can converge strongly in $L^1(X)$, contradicting $(2)$.
\end{proof}

We conclude with some comments on the isoperimetric profile at infinity $I^\infty_X$.

\begin{remark}
Let $(X,\dist,\mathcal{H}^N)$ be a noncompact $\RCD(K,N)$ space. Assume there exists $v_0>0$ such that $\mathcal{H}^N(B_1(x))\geq v_0$ for every $x\in X$. Let
\[
\mathcal{F}_\infty \eqdef \left\{ (Y,\dist_Y,\haus^N_Y) \st \text{$Y$ is a pmGH-limit at infinity along $X$} \right\},
\]
and let $X^\infty\eqdef \bigsqcup \mathcal{F}_\infty$ be the disjoint union of the spaces at infinity along $X$. Defining
\[
I_{X^\infty} (V) \eqdef \inf \left\{\sum_{i \in \N} P(E_i) \st E_i \subset Y_i \in \mathcal{F}_\infty \,\forall\, i , \quad \sum_{i\in\N} \haus^N(E_i)=V \right\},
\]
for $V\ge0$, one can prove that $I^\infty_X =I_{X^\infty}$. Indeed, the equality follows arguing similarly as in the proof of \cref{prop:ProfileDecomposition}.

\medskip

In case hypotheses on the asymptotic behavior of the space $X$ are assumed, for instance in case the space is GH-asymptotic to some limit (see \cite{MondinoNardulli16, EichmairMetzger, Shi, AFP21}), then one clearly can identify $I^\infty_X$ more explicitly.
\end{remark}

\printbibliography[title={References}]
\end{document}